\documentclass[onecolumn,3p]{elsarticle}

\usepackage[utf8]{inputenc}
\usepackage[T1]{fontenc}

\usepackage{lineno,hyperref}
\modulolinenumbers[50]

\usepackage{array}
\usepackage{color}
\usepackage{multirow}
\usepackage{graphicx}
\usepackage{subcaption}
  \usepackage{scrextend}

\usepackage{algorithmicx}
\usepackage{algpseudocode}
\usepackage{algorithm}

\usepackage{bbm}
\usepackage[table]{xcolor}

\usepackage{mathrsfs}
\usepackage{epsfig}
\usepackage{amssymb}
\usepackage{graphicx}
\usepackage{fancyhdr}
\usepackage{float}
\usepackage{epstopdf}

\usepackage{hyperref}

\hypersetup{
     linkcolor = blue,
     citecolor = blue,
     urlcolor = blue,
     anchorcolor = blue
    }

\usepackage{amsthm}

\usepackage{tikz}

\usepackage{pgfplots}
\pgfplotsset{compat=1.11}

\usetikzlibrary{arrows}

\usepackage{xspace}
\usepackage{float}
\usepackage{bbm}
\usepackage{amssymb}
\usepackage{amsmath}
\usepackage{amsthm}

\newcommand{\ie}{\textit{i.e.,}\xspace}

\newtheorem{theorem}{Theorem}[section]

\newtheorem{lemma}[theorem]{Lemma}

 \newtheorem{remark}[theorem]{Remark}
\newtheorem{corr}[theorem]{Corollary}
\newtheorem{conj}[theorem]{Conjecture}

\newcommand{\Z}[0]{\mathbf{Z}}
\newcommand{\e}[0]{\mathbf{e}}
\newcommand{\p}[0]{\mathbf{p}}
\newcommand{\q}[0]{\mathbf{q}}
\newcommand{\C}[0]{\mathbf{C}}
\newcommand{\X}[0]{\mathbf{X}}
\newcommand{\Y}[0]{\mathbf{Y}}

\newcommand{\E}[0]{\mathbb{E}}
\newcommand{\PP}[0]{\mathbf{P}}

\definecolor{lightgray}{gray}{0.9}

\definecolor{ltgray}{RGB}{200,200,200}
\definecolor{dkgray}{RGB}{150,150,150}

\newcommand{\todo}[1]{}

\numberwithin{equation}{section}
 
\begin{document}

\begin{frontmatter}

\title{Conditional gambler's ruin problem  with arbitrary winning and losing probabilities
with applications}

\author[uni]{Paweł Lorek}
 \ead{Pawel.Lorek@math.uni.wroc.pl}
\author[uni]{Piotr Markowski}
\ead{Piotr.Markowski@math.uni.wroc.pl}

%
%
%
 \address[uni]{Mathematical Institute, University of Wrocław, pl. Grunwaldzki 2/4, 50-384, Wrocław, Poland}

%
%
%
%
%

\begin{abstract}
In this paper we provide formulas for the expectation of a conditional game duration 
in a finite state-space one-dimensional gambler's ruin problem 
with arbitrary winning $p(n)$ and losing $q(n)$ probabilities (\ie they depend on the 
current fortune). The formulas are stated in terms of the parameters of the system.
Beyer and Waterman [\textsl{Mathematics Magazine}, \textbf{50}(1), 1977]
showed that for the classical gambler's
ruin problem the distribution of a conditional absorption time is symmetric in $p$ and $q$.
Our formulas imply that for non-constant winning/losing probabilities
the expectation of a conditional game duration is symmetric in these probabilities (\ie it is the same if
we exchange $p(n)$ with $q(n)$) as long as a ratio  $q(n)/p(n)$ is constant. 

Most of the formulas are applied to a non-symmetric random walk on a circle/polygon.
Moreover, for a symmetric random walk on a circle we construct an optimal strong stationary dual chain -- which turns out to be an absorbing, non-symmetric, birth and death chain. 
We apply our results and provide a formula for its expected absorption time, which is a 
fastest strong stationary time for  the aforementioned symmetric random walk on a circle. This way we improve upon a result of Diaconis and Fill [\textsl{The Annals of Probability}, \textbf{18}(4), 1990], where strong stationary time -- however not the fastest -- was constructed. Expectations of the fastest strong stationary time and the one constructed by Diaconis and Fill differ by 3/4, independently of a circle's size.

\end{abstract}


\end{frontmatter}
\linenumbers


 \section{Introduction} 
 The classical gambler's ruin problem is following.
Having initially $i$ dollars, $1\leq i\leq N-1$, in one step we either win one dollar (\ie we move to 
$i+1$) with probability $p\in(0,1)$, or we lose one dollar (\ie we move to $i-1$) with probability $q=1-p$.
The game ends when the player reaches $N$ (wins the game) or 0 (goes broke).
The typical questions one can ask are: 
\begin{itemize}
\item What is the probability of winning (\ie reaching $N$ before $0$)?
\item What is the (expected) game duration?
\item What is the (expected) conditional game duration (\ie game duration given we win or given we lose)?
\item Is the (expected) conditional game duration symmetric in $p$ and $q$?
\end{itemize}

Similarly, one can consider random walk on $\Z_{m+1}=\{0,\ldots,m\}$:
being at state $i$
we either move clockwise with a probability $p\in(0,1)$ (\ie from $i$ to $i+1 \bmod (m+1)$)
or we move counterclockwise  with a probability $1-p$ (\ie we move from $i$ to  $i-1 \bmod (m+1)$).
We will refer to this as to \textsl{the classical random walk on a 	polygon} (cf. \cite{Sarkar2006}).
Assuming we start at $i$, the typical questions one can ask are: 
\begin{itemize}
\item What is the probability that all vertices have been visited before the particle returns to $i$?
\item What is the probability that the last vertex visited is $j$ ?
\item What is the expected number of moves needed to visit all the vertices?
\item What is the expected additional number of moves needed to return to $i$ after visiting all the vertices?
\end{itemize}
All above questions were answered in the classical settings. Several generalizations were studied.
The probability of winning in a gambler's ruin problem with general winning and losing probabilities (\ie $p(i)$ being 
probability of moving from $i$ to $i+1$ and $q(i)$ being the probability of moving from $i$ to $i-1$, with 
$p(i)+q(i)\leq 1$, $i\in\{1,\ldots,N-1\}$) goes back to Parzen \cite{Parzen62}, revisited in \cite{El-Shehawey2009}.
Siegmund duality based proof is given in \cite{2015Lorek_gambler}  (where more general, multidimensional, game is considered).
In \cite{Lengyel2009a} the questions related to the conditional game duration are answered for the classical gambler's ruin 
problem with ties allowed, \ie $p+q \leq 1$ (with probability $1-(p+q)$ we can stay at a given state). 
In \cite{Lefebvre2008} author considers specific generalization, namely $p(i)=q(i)={1\over 2(2ci+1)}, c\geq 0 $ 
(thus the probability of staying is $1-{1\over 2ci+1}$) and answers the question about the winning probability and the expected game duration
(and also considers the corresponding diffusion process). 
In this paper we present formulas for the  expected (conditional) absorption time in terms of 
parameters of the system (\ie winning/losing probabilities $p(i), q(i)$).
Similar problem was considered in \cite{El-Shehawey2000}, the recursion 
for the expected conditional game duration is given therein 
(equations (3.4) and (3.5)), however it is not solved in its general form --
later on author considers only constant winning/losing probabilities.
In  \cite{Gong2012} (similar results with different proofs are presented in \cite{Mao2016})
the generating function of absorption time (including a conditional one) is given in terms 
of eigenvalues of a transition matrix and eigenvalues of a truncated transition matrix.
The questions for the classical random walk on a polygon were answered 
in \cite{Sarkar2006}. Some generalizations (rather then allowing arbitrary winning/losing probabilities,
symmetric random walks on tetrahedra, octahedra, and hexahedra, are considered)
are studied  in \cite{Sarkar2017}.
\medskip\par 
In 1977 in \cite{Beyer77} it was shown 
that for a classical gambler's ruin problem with $p(n)=p=1-q(n)=1-q$, the distribution
of a conditional game duration is symmetric in $p$ and $q$, \ie it is the same as in a game 
with $p'=q$ and $q'=p$. In 2009 in \cite{Lengyel2009a} it was extended to a case $p+q<1$ 
(\ie the classical case with ties allowed). In this paper we show that 
that the expected conditional game duration is symmetric also for 
non-constant winning/losing probabilities  $p(n), q(n)$ as long as $q(n)/p(n)$ is constant
(thus, including for example the spatially non-homogeneous case).

\medskip\par 
In Section \ref{sec:gambler} we introduce gambler's ruin problem with arbitrary winning and losing probabilities $p(i), q(i)$
together with main results. In Section \ref{sec:gam_const_r} the main result is applied to constant $r(i)=r=q(i)/p(i)$,
in Section \ref{sec:gam_spat} it is applied to non-homogeneous case, whereas the classical case 
is recalled in Section \ref{sec:gambl_class}. The main example is given in Section \ref{sec:gam_example}.
The results are applied to a random walk on polygon 
in Section \ref{sec:polygon}. Last Section \ref{sec:proofs} contains proofs of main results.

  \section{Gambler's ruin problem}\label{sec:gambler}
 \noindent 
Fix an integer $N\geq 2$. Let 
 $$\p=(p(0), p(1),\ldots, p(N)), \quad \q=(q(0), q(1),\ldots, q(N)),$$
 where $p(0)=q(0)=p(N)=q(N)=0$ and $p(i),q(i)>0, p(i)+q(i)\leq 1$ for $i \in \{1,2\ldots,N-1\}$. 
Consider a Markov chain $\mathbf{X}=\{X_k\}_{k\geq 0}$ on $\E =\{0,1,\ldots,N\}$ with transition probabilities
$$\PP_{X}(i,j) 
=\left\{
\begin{array}{lllll}
 p(i) & \textrm{if} & j=i+1, \\[4pt]
 q(i) & \textrm{if} & j=i-1, \\[4pt]
 1-(p(i)+q(i)) & \textrm{if} & j=i. \\[4pt]
\end{array}
\right.
$$
\par
We will refer to $\X$ starting at $i$ as to the (gambler's ruin) \textsl{game} $G(\p,\q,0,i,N)$. 
Note that the chain will  eventually end up in either in $N$ (the \textsl{winning} state)  
or in $0$ (the \textsl{losing} state).
To simplify some notation, let $r(i)={q(i)\over p(i)}$ for $i \in \{1,\ldots,N-1\}$.

Define $\tau_j=\inf\{k: X_k  =j\}$.
We will study the following \textsl{smaller} games $G(\p,\q,j,i,k)$ with  $k$ as the 
\textsl{winning} state  and $j$ as the \textsl{losing} state ($j\leq i\leq k$), 
\ie $p(j)=q(j)=p(k)=q(k)=0$.
Let us define:
\begin{eqnarray*}
 \rho_{j:i:k}& = & P(\tau_k<\tau_j | X_0=i), \\[8pt] 
 T_{j:i:k}& = & \inf\{n\geq 0: X_n=j \textrm{ or } X_n=k,\ X_0=i\},\\[8pt]
 W_{j:i:k} & = & T_{j:i:k} \textrm{ conditioned on } X_{T_{j:i:k}} = k,  \\[8pt]
 B_{j:i:k} & = & T_{j:i:k} \textrm{ conditioned on } X_{T_{j:i:k}} = j.  \\
\end{eqnarray*}
%
%

In other words:   $\rho_{j:i:k}$ is the probability that a gambler starting with $i$ dollars wins in 
the smaller game;
$T_{j:i:k}$ is the distribution of a game   duration 
(time till gambler either wins or goes broke); $W_{j:i:k}$ is the distribution of $T_{j:i:k}$ conditioned on $X_{T_{j:i:k}}=k$ (winning)
and similarly $B_{j:i:k}$ is the distribution of $T_{j:i:k}$ conditioned on $X_{T_{j:i:k}}=j$ (losing).  
\medskip\par 
\textbf{Notation.} For given rates $\p, \q$ by $\p\leftrightarrow\q$ we understand
new rates $\p'=\q, \q'=\p$. For some random variable $R$ (one of $\rho, T, W, B$)
for a game with rates $\p, \q$, by $R(\p\leftrightarrow \q)$ we understand the random variable
defined for a game with rates $\p'=\q, \q'=\p$ (and similarly, e.g., $ER(\p\leftrightarrow\q)$ is 
an expectation of $R$ defined for such a game). We say that  $R$ ($ER$)
 is \textsl{ symmetric in} $\p$ and $\q$ if  $R\stackrel{distr}{=}R(\p\leftrightarrow\q)$ ($ER=ER(\p\leftrightarrow\q)$).
 \medskip\par
 By $f(n)=\Theta(g(n))$ we mean $\exists (c_1,c_2>0)$ $\exists (n_0)$ $ \forall (n>n_0)$ $c_1 g(n)\leq f(n) \leq c_2 g(n)$.
In this section we  use the convention: empty sum equals $0$, empty product equals $1$;
however in Section \ref{sec:polygon} we use some nonstandard notation, 
see details on page  \pageref{polygon:notation}.

\medskip\par 
In next theorem we provide formulas for expected game duration, for completeness (and since we will need them later)
 we also include known results for $\rho_{j:i:k}$.
 
%
%
\begin{theorem}\label{thm:main_rho_ET}
Consider the gambler's ruin problem on $\E=\{0,1,\ldots,N\}$ described above. We have
\begin{eqnarray}
 \rho_{j:i:k} & = & {\displaystyle  \sum_{n=j+1}^i \prod_{s=j+1}^{n-1}\left({q(s)\over p(s)}\right) \over 
 \displaystyle   \sum_{n=j+1}^k \prod_{s=j+1}^{n-1}\left({q(s)\over p(s)}\right)} = {\displaystyle  \sum_{n=j+1}^i \prod_{s=j+1}^{n-1}r(s) \over 
 \displaystyle   \sum_{n=j+1}^k \prod_{s=j+1}^{n-1}r(s)}, \nonumber \\[8pt]
  ET_{j:i:k} & = &
  \frac{\sum_{n=j+1}^{k-1}[d_n\sum_{s=j+1}^n \frac{1}{p(s)d_s}]}{\sum_{n=j}^{k-1}d_n}\sum_{n=j}^{i-1}d_n-\sum_{n=j+1}^{i-1}\left[d_n\sum_{s=j+1}^n \frac{1}{p(s)d_s}\right],\label{eq:main_gambler_ET}
 \end{eqnarray} 

where $d_s=\prod_{i=j+1}^s {q(i)\over p(i)}=\prod_{i=j+1}^s r(i)$ (with convention $d_j=1$).

\end{theorem}
\noindent 
The proof of Theorem \ref{thm:main_rho_ET} is postponed to Section \ref{sec:proof_gambler_rho_ET}.
We will also need a formula for $ET_{j:i:k}$ in case when  $k$ is the only absorbing state.
\begin{theorem}\label{thm:ETjik_oneAbs} Fix $j\leq i\leq k$ and consider a birth and death chain on $\{j,\ldots,k\}$ 
with rates $p(s), q(s), s=j,\ldots,k$ with $q(j)=p(k)=q(k)=0$ and $q(s)>0 $ for $ s=j+1,\ldots,k-1$
and $p(s)>0$ for $s=j,\ldots,k-1$ (\ie $k$ is the only absorbing state). Then, the expectation of 
absorption time, starting from $i$ is given by
$$ET_{j:i:k}=\sum_{n=i}^{k-1}\left[d_n\sum_{s=j}^n \frac{1}{p(s)d_s}\right].$$
\end{theorem}

%
%
%
Now we go back to situation with two absrobing states, \ie also $p(j)=0$.
Next theorem (our main contribution) gives the formulas for  $EW_{0:i:k}$ and $EB_{0:i:k}$. 
First, let us introduce some necessary  notation. With some abuse of notation let us extend  

 $$ \rho_{j:i:k}  =  {\displaystyle  \sum_{n=j+1}^i \prod_{s=j+1}^{n-1}\left({q(s)\over p(s)}\right) \over 
 \displaystyle   \sum_{n=j+1}^k \prod_{s=j+1}^{n-1}\left({q(s)\over p(s)}\right)} = {\displaystyle  \sum_{n=j+1}^i \prod_{s=j+1}^{n-1}r(s) \over 
 \displaystyle   \sum_{n=j+1}^k \prod_{s=j+1}^{n-1}r(s)}$$
for $k<i$ (but still $k> j$). Note that in such a case we may have $\rho_{j:i:k},$ thus this 
has no interpretation in terms of probability anymore.

For given integers $n,m,k$ such that $n\leq m, k\in\{0,{\lfloor (m-n+1)/2 \rfloor}\}$ define 
\begin{equation}\label{eq:j}
\mathbf{j}^{n,m}_{k}=\big\{\{j_1,j_2,\ldots,j_k\}: j_1 \geq n+1,  j_k \leq m, j_i \leq j_{i+1}-2 \textrm{ for } 1\leq i\leq k-1 \big\}. 
\end{equation}

For  given $\p, \q$ 
and  
$\mathbf{j}\in \mathbf{j}^{n,m}_{k}$ 
define   
\medskip\par 
\begin{equation}\label{eq:delta}
 \delta^{n,m}_\mathbf{j}=(-1)^k\prod_{s \in j}r(s)\prod_{s \in \{n,\ldots,m\} \setminus \mathbf{j}\cup (\mathbf{j}-1)  } 1+r(s),
\end{equation}
where $\{n,\ldots,m\}$ is an empty set for $n>m$ and $\mathbf{j}-1=\{j_1-1,j_2-1,\ldots,j_k-1\}$ for $\mathbf{j}=\{j_1,j_2,\ldots,j_k\}$.
Finally, let 
\begin{equation}\label{eq:xi}
\xi^{n,m}_k=\sum_{\mathbf{j} \in \mathbf{j}^{n,m}_k}\delta^{n,m}_\mathbf{j}. 
\end{equation}

\noindent 
Now we are ready to state our main theorem.
\begin{theorem}\label{thm:main_cond}
Consider the gambler's ruin problem on $\E=\{0,1,\ldots,N\}$ described 
above. We have
\begin{eqnarray}  
EW_{0:i:N} & =& EW_{0:1:N}-EW_{0:1:i}, \ \  \textrm{where} \label{eq:main_gambler_EWiN}\\[8pt]
 EW_{0:1:i} & = & \displaystyle\sum_{n=1}^{i-1} \frac{\rho_{0:n:i}}{p(n)} \sum_{s=0}^{\lfloor (i-1-n)/2 \rfloor}\xi^{n+1,i-1}_s.
 \label{eq:main_gambler_EW1i}
\end{eqnarray} 
Moreover, we have
\begin{eqnarray}
  EB_{0:i:N} & = & EW'_{0:N-i:N}\label{eq:main_gambler_EB},
\end{eqnarray} 
where $W'_{0:N-i:N}$ is defined for a gambler's ruin problem with rates $p'(i)=q(N-i)$ and $q'(i)=p(N-i)$ for $i \in \E$.
\end{theorem}
\noindent
The proof of Theorem \ref{thm:main_cond}  is  postponed to Section \ref{sec:proof_gambler}.
\medskip \par 

\subsection{Constant $r(n)=r={q(n)\over p(n)}$}\label{sec:gam_const_r}
In this section we will apply Theorems \ref{thm:main_rho_ET} and \ref{thm:main_cond}  to a 
gambler's ruin problem with  constant $r={q(i)\over p(i)}$. 
The winning probabilities $\rho_{0:i:N}$ are known
(they are the same as in the classical formulation of the problem), we will focus on a game duration.
We have 
\begin{corr}\label{cor:ET_constr}
 Consider the gambler's ruin problem on $\E=\{0,\ldots,N\}$ with 
 constant $r={q(i)\over p(i)}$.  We have 
 $$
 \begin{array}{llll}
  r = 1: & ET_{j:i:k} & = & \displaystyle
 {i-j\over k-j} \sum_{n=j+1}^{k-1}\sum_{s=j+1}^n \frac{1}{p(s)} -\sum_{n=j+1}^{i-1}\sum_{s=j+1}^n \frac{1}{p(s)},\label{eq:main_gambler_ET_const_r1}\\[21pt] 
       & ET_{0:i:N} & = & \displaystyle
 {i\over N} \sum_{n=1}^{N-1}\sum_{s=1}^n \frac{1}{p(s)} -\sum_{n=1}^{i-1}\sum_{s=1}^n \frac{1}{p(s)},\label{eq:main_gambler_ET_const_r1b}\\[21pt] 
  r\neq 1:& ET_{j:i:k} & = & \displaystyle
 {r^j-r^i\over r^j-r^k} \sum_{n=j+1}^{k-1}\left[r^{n}\sum_{s=j+1}^n \frac{r^{-s}}{p(s)}\right] -\sum_{n=j+1}^{i-1}\left[r^{n}\sum_{s=j+1}^n \frac{r^{-s}}{p(s)}\right],\label{eq:main_gambler_ET_const_r}\\[21pt] 
    & ET_{0:i:N} & = & \displaystyle
 {1-r^i\over 1-r^N} \sum_{n=1}^{N-1}\left[r^{n}\sum_{s=1}^n \frac{r^{-s}}{p(s)}\right] -\sum_{n=1}^{i-1}\left[r^{n}\sum_{s=1}^n \frac{r^{-s}}{p(s)}\right].\label{eq:main_gambler_ET_const_rb}
 \end{array}
 $$
%
\end{corr}
\smallskip\par
\begin{proof}
 We have $d_k=\prod_{j=1}^k r=r^k$. Simple recalculations of 
 (\ref{eq:main_gambler_ET}) yield the result.
\end{proof}

\medskip\par 
\noindent
For constant $r$ we have that  $\delta_\mathbf{j}^{n,m}$ (given 
in (\ref{eq:delta})) for all $i \in \{1,\ldots,N-1\}$ 
 depends on $\mathbf{j}$ only through $k$, thus 
 \begin{equation}\label{eq:xi_const_r}
 \xi^{n,m}_k=\sum_{\mathbf{j} \in \mathbf{j}^{n,m}_k}\delta^{n,m}_\mathbf{j}=C_k^{n,m}(-r)^k (1+r)^{m+1-n-2k}, 
 \end{equation}
  where $C_k^{n,m}=|\mathbf{j}^{n,m}_k|$. 
Moreover, we have $|\mathbf{j}^{n,m}_k|=T(m+1-n,k)$, where $T(n,k)={ n-k \choose k}$ is the number of subsets of 
$\{1,2,...,n-1\}$ of size $k$  containing no consecutive integers
\footnote{\url{http://oeis.org/A011973}}.
 \medskip\par 
 \noindent 
 The proof of the next corollary requires the following lemma.
\begin{lemma}\label{lem:sum}
Let $n\in\mathbb{N}$ and $r\geq 0$. We have 
\begin{equation}\label{eq:lem_sum}
\sum_{k=0}^n {n-k\choose k} \left(-{r\over (1+r)^2}\right)^k  =
\left\{ 
\begin{array}{lllll}
\displaystyle {1-r^{n+1}\over (1+r)^n(1-r)} & \textrm{ if } r\neq 1, \\[16pt]
\displaystyle {n+1\over 2^n} & \textrm{ if } r=1.
\end{array}
\right.
\end{equation}
\end{lemma}
\noindent
The proof of Lemma \ref{lem:sum} is given in Section \ref{sec:proof_gambler}.

\begin{remark}\rm
 Note that the assertion of Lemma \ref{lem:sum} can be stated in the following form 
 (simply substituting $c={r\over (1+r)^2}$): for $n\in\mathbb{N}$ and 
 $c\in(0,1/4]$ we have 
   $$\sum_{k=0}^n {n-k\choose k} \left(-c\right)^k = 
   \left\{ 
   \begin{array}{lllll}
   \displaystyle {1-\gamma^{n+1}\over (1+\gamma)^n(1-\gamma)},&
   \textrm{where }\ \displaystyle \gamma={1-2c+\sqrt{1-4c}\over 2c}, \ \textrm{if}  \ c\in(0,1/4), \\[14pt]
   \displaystyle {n+1\over 2^n} &  \textrm{ if } c=1/4.
   \end{array}\right. 
   $$
   
   \noindent
   These sums for $c\in\{-1,1\}$ were known ($F(n)$ is the $n$-th Fibonacci number):
   $$
   \begin{array}{llll}
    \displaystyle\sum_{k=0}^n {n-k\choose k} & = & F(n+1), \\[30pt]
    \displaystyle\sum_{k=0}^n {n-k\choose k} (-1)^k& = & 
	    \left\{
	     \begin{array}{llll}
	      1 & \textrm{\ if \ } n \bmod 6 \in\{0,1\},\\[10pt]
	      0 & \textrm{\ if \ } n \bmod 6 \in\{2,5\},\\[8pt]
	      -1 & \textrm{\ if \ } n \bmod 6 \in\{3,4\}.
	      \end{array}
	      \right.
   \\    
   \end{array}
$$   
\end{remark}
\smallskip\par\noindent 
We will give formulas for $EW_{0:1:i}$ for several cases 
($EW_{0:i:N}$ can be calculated via (\ref{eq:main_gambler_EWiN})).
\begin{corr}\label{cor:cons_r_EW}
Consider the gambler's ruin problem on $\E=\{0,\ldots,N\}$ with 
constant $r={q(i)\over p(i)}$. We have:
\begin{eqnarray}  
 r=1: \quad  EW_{0:1:i} & = &  \displaystyle\sum_{n=1}^{i-1} \frac{\rho_{0:n:i}}{p(n)} \sum_{s=0}^{\lfloor (i-1-n)/2 \rfloor}\xi^{n+1,i-1}_s  =\sum_{n=1}^{i-1} \frac{n/i}{p(n)} (i-n).\nonumber\\[10pt]
 r\neq 1: \quad EW_{0:1:i} & = &  \displaystyle\sum_{n=1}^{i-1} \frac{\rho_{0:n:i}}{p(n)} \sum_{s=0}^{\lfloor (i-1-n)/2 \rfloor}\xi^{n+1,i-1}_s = \sum_{n=1}^{i-1}  \frac{ \frac{1-r^n}{1-r^i} (1-r^{i-n})}{p(n)(1-r)}.\label{eq:EW01i_nonconstr}
\end{eqnarray} 
Additionally, if $p(n)=p$ is constant (so is $q(n)$ then, since $r(n)$ is constant) we have
\begin{eqnarray}  
 r=1:  \quad  EW_{0:1:i} & = &  \displaystyle   {1\over p}\sum_{n=1}^{i-1} \frac{n}{i} (i-n)={(i-1)(i+1)\over 6p},\label{eq:EW01i_constr}\\[10pt]
 r\neq 1: \quad EW_{0:1:i} & = &  \displaystyle
  {1\over p}\sum_{n=1}^{i-1}  \frac{ \frac{1-r^n}{1-r^i} (1-r^{i-n})}{1-r}=\frac{1}{p(1-r^i)(1-r)}\displaystyle\sum_{n=1}^{i-1} (1-r^n) (1-r^{i-n}) \nonumber\\
  & = & \frac{i(1+r^i)-(1+r)\frac{1-r^i}{1-r}}{p(1-r^i)(1-r)} = \frac{1}{p(1-r)}\left(i\frac{1+r^i}{1-r^i} -\frac{1+r}{1-r}\right). \nonumber
\end{eqnarray} 
\end{corr}
\begin{proof}
We will only show case $r=1$, general $p(n)$ (the proof for $r\neq 1$ is very similar).
Let us calculate $\xi_s^{n+1,i-1}$ first. From 
(\ref{eq:xi_const_r}) and form of $C_k^{n,m}$ for $r=1$ we have
$$\xi_s^{n+1,i-1}  =  C_s^{n+1,i-1} (-1)^s 2^{i-n-1-2s} = 2^{i-n-1}{i-n-1-s\choose s} \left(-{1\over 4}\right)^s.
$$
From  Theorem \ref{thm:main_cond} (eq. (\ref{eq:main_gambler_EW1i})) and the fact that $\rho_{0:n:i}={n/i}$ (since $r=1$) 
we have 
$$
\begin{array}{lclll}
EW_{0:1:i} & = & \displaystyle \sum_{n=1}^{i-1}{n/i\over p(n)} \sum_{s=0}^{\lfloor (i-1-n)/2 \rfloor}\xi^{n+1,i-1}_s \\
  & = & \displaystyle  \sum_{n=1}^{i-1}{n/i\over p(n)}  2^{i-n-1} \sum_{s=0}^{\lfloor (i-1-n)/2 \rfloor}{i-n-1-s\choose s} \left(-{1\over 4}\right)^s \\
  & \stackrel{\textrm{Lemma}\ \ref{lem:sum}}{=} & \displaystyle  \sum_{n=1}^{i-1}{n/i\over p(n)}  2^{i-n-1} {i-n-1+1\over 2^{i-n-1} } = 
 \sum_{n=1}^{i-1}{n/i\over p(n)}(i-n),
\end{array}
$$
what finishes the proof.
\end{proof}

In 1977 Beyer and Waterman \cite{Beyer77} showed that for a classical case
\ie for constant birth $p(n)=p$ and death $q(n)=q$ rates such that 
$p+q=1$, the distribution of $W_{0:i:N}$ is symmetric 
in $p$ and $q$ (\ie it has the same distribution for birth rate 
$p'=q$ and death rate $q'=p$). In 2009 Lengyel \cite{Lengyel2009a} showed that 
this holds also for the classical case with ties allowed, \ie 
$p+q<1$. In the following theorem we show that $EW_{0:i:N}$
is symmetric in $\p$ and $\q$ (\ie it is the same for 
case with birth deaths $p'(n)=q(n)$ and death rates $q'(n)=p(n)$) as long 
as $r(n)={q(n)\over p(n)}$ is constant.
\begin{theorem}
Consider the gambler's ruin problem on $\E=\{0,\ldots,N\}$ with 
constant $r={q(i)\over p(i)}$. 
We have 
$$EW_{0:i:N} = EW_{0:i:N}(\p \leftrightarrow \q),$$
(\ie $EW_{0:i:N}$ is symmetric in $\p$ and $\q$).
\end{theorem}

\begin{proof}
By (\ref{eq:main_gambler_EWiN}) it is enough to show that 
$EW_{0:1:i} = EW_{0:1:i}(\p \leftrightarrow \q)$.
\par 
Let $W_{0:1:i}$ be defined for rates $\p$ and $\q$, whereas 
$W'_{0:1:i}$ be defined for rates $\p'=\q$ and $\q'=\p$,
thus $r'=1/r$. Since $r={q(n)\over p(n)}, $ we have $p'(n)=q(n)=rp(n)$. 
\par 

\begin{eqnarray*}  
 EW'_{0:1:i} & = & \displaystyle\sum_{n=1}^{i-1} \frac{1}{p'(n)}\frac{(1-\frac{1}{r^n})}{(1-\frac{1}{r^i})}\frac{  (1-\frac{1}{r^{i-n}})}{(1-\frac{1}{r})} =\sum_{n=1}^{i-1} \frac{1}{rp(n)}\frac{r^i(1-r^n)}{r^n(1-r^i)}\frac{r(1-r^{i-n})}{r^{i-n}(1-r)}\\
& = & \displaystyle\sum_{n=1}^{i-1} \frac{1}{p(n)}\frac{(1-r^n)}{(1-r^i)}\frac{  (1-r^{i-n})}{(1-r)},
\end{eqnarray*}
what is equal to (\ref{eq:EW01i_nonconstr}).
\end{proof}
\noindent
It is natural to state the following conjecture. 
\begin{conj}
 Consider the gambler's ruin problem on $\E=\{0,\ldots,N\}$ with 
constant $r={q(i)\over p(i)}$. Then, the distribution of $ W_{0:i:N}$ is symmetric in $\p$ and $\q$.
\end{conj}

\subsection{The spatially non-homogeneous  case}\label{sec:gam_spat}
In this Section we consider gambler's ruin problem with 
birth rates $p(n)={p\over 2 c n +1}$ and death rates $q(n)={q\over 2 c n+1}$,
where $c$ is a non-negative constant. This is often called the spatially non-homogeneous gambler's ruin problem. We will thus still consider 
case with constant $r(n)$, but with specific rates. 
As far as we are aware, all results in this section, except the one for $p(n)=q(n)=1/2$,  
are new.

\begin{corr}
Consider the spatially  non-homogeneous  gambler's ruin problem. We have 
 \begin{eqnarray*}
  r = 1: ET_{0:i:N} & = &
  \frac{1}{2p}\left(i N\left(1+{2c\over 3}N\right) - i^2\left(1+{2c\over 3}i\right)\right),\\[8pt]  
  r\neq 1: ET_{0:i:N} & = &
  \frac{1}{p(r-1)}\left(\frac{1-r^i}{1-r^N}\left( -cN^2-N\frac{(cr+c)}{r-1}-N  \right) +ci^2+i\frac{(cr+c)}{r-1}+i    \right).\\[8pt]  
 \end{eqnarray*}
\end{corr}

\begin{proof}
Applying Corollary \ref{cor:ET_constr} we have:
\begin{itemize}
 \item Case $r=1$
$$
\begin{array}{llll}
   ET_{0:i:N} & = & \displaystyle \frac{i}{N} \sum_{n=1}^{N-1}\sum_{s=1}^n \frac{1}{p(s)} -\sum_{n=1}^{i-1}\sum_{s=1}^n \frac{1}{p(s)} = \displaystyle \frac{i}{N} \sum_{n=1}^{N-1}\sum_{s=1}^n \frac{2cn+1}{p} -\sum_{n=1}^{i-1}\sum_{s=1}^n \frac{2cn+1}{p}\\[20pt]
 & = & \displaystyle \frac{1}{p}\left( \frac{i}{N} \sum_{n=1}^{N-1}n(cn+c+1) -\sum_{n=1}^{i-1}n(cn+c+1) \right)\\[20pt]
 & = & \displaystyle \frac{1}{p}\left( \frac{i}{N} \frac{1}{6}(N-1)(N(2c(N+1)+3) -\frac{1}{6}(i-1)(i(2c(i+1)+3) \right)\\[20pt]
& = & \displaystyle \frac{1}{2p}\left(i N\left(1+{2c\over 3}N\right) - i^2\left(1+{2c\over 3}i\right)\right).
\end{array}
$$
\item Case $r\neq 1$

\begin{eqnarray*}
ET_{0:i:N}
& = &  {1-r^i\over 1-r^N} \sum_{n=1}^{N-1}\left[r^{n}\sum_{s=1}^n \frac{r^{-s}}{p(s)}\right] -\sum_{n=1}^{i-1}\left[r^{n}\sum_{s=1}^n \frac{r^{-s}}{p(s)}\right]\\
& = & \frac{1}{p}\left( {1-r^i\over 1-r^N} \sum_{n=1}^{N-1}\left[r^{n}\sum_{s=1}^n r^{-s}(2cs+1)\right] -\sum_{n=1}^{i-1}\left[r^{n}\sum_{s=1}^n r^{-s}(2cs+1)\right] \right). 
\end{eqnarray*}
%
We have
 $$\sum_{s=1}^n r^{-s}(2cs+1)=\frac{r^{-n}}{(r-1)^2}\left( 2cr^{n+1}-2cnr+2cn-2cr+r^{n+1}-r^n-r+1 \right)$$
 and
 $$
 \begin{array}{llllllll}

 \multicolumn{6}{l}{\displaystyle\sum_{n=1}^{k-1}\left[r^{n}\frac{r^{-n}}{(r-1)^2}\left( 2cr^{n+1}-2cnr+2cn-2cr+r^{n+1}-r^n-r+1 \right)\right]}\\[20pt]

 & = & \displaystyle & \displaystyle\frac{1}{(r-1)^2}&\Big(-\frac{2cr(r-r^k)}{r-1}-c(k-1)kr+c(k-1)k-2cr(k-1)   \\[20pt]
 &   &               & &\displaystyle \ \ \ \left.-\frac{r(r-r^k)}{r-1}+\frac{r-r^k}{r-1}+r-kr+k-1   \right)\\[20pt]
 &= &\multicolumn{6}{l}{\displaystyle \frac{1}{(r-1)^2}\left(-ck^2(r-1)+\frac{(2cr+r-1)(r^k-1)}{r-1}-k(cr+c+r-1)\right)}.
 \end{array}
 $$
%
Thus,
 $$
 \begin{array}{llllllllll}

ET_{0:i:N} &= & \displaystyle \frac{1}{p(r-1)^2}\Big\{&  \displaystyle \frac{1-r^i}{1-r^N} \Big( & -cN^2(r-1)+\frac{(2cr+r-1)(r^N-1)}{r-1}-N(cr+c+r-1)\Big)\\[16pt]
& & & &-\left(-ci^2(r-1)+\frac{(2cr+r-1)(r^i-1)}{r-1}-i(cr+c+r-1)\right)\Big\}\\[16pt]

   &= & \displaystyle \frac{1}{p(r-1)^2}\Big\{&  \displaystyle    \frac{1-r^i}{1-r^N}\Big( & \displaystyle  -cN^2(r-1)-N(cr+c+r-1)\Big)\\[16pt]
& & & & \displaystyle  +ci^2(r-1)+i(cr+c+r-1) \Big\}\\[16pt]
& = & \multicolumn{6}{l}{\displaystyle \frac{1}{p(r-1)}\left(\frac{1-r^i}{1-r^N}\left( -cN^2-N\frac{(cr+c)}{r-1}-N  \right) +ci^2+i\frac{(cr+c)}{r-1}+i\right),}\\[16pt]
  \end{array}
 $$
what was to be shown.
\end{itemize}
%
\end{proof}

\begin{remark}
 \rm Note that for $p(n)=q(n)=1/2$ we have $ET_{0:i:N} = i N\left(1+{2c\over 3}N\right) - i^2\left(1+{2c\over 3}i\right),$
 \ie we obtained Proposition 2.1 from \cite{Lefebvre2008}.
\end{remark}

\medskip\par 

Concerning the conditional game duration (because of 
(\ref{eq:main_gambler_EB}) it is enough to provide formula only for $EW_{0:i:N}$)
we have
\begin{corr}\label{cor:spat_hom}
Consider the spatially  non-homogeneous gambler's ruin problem.  We have 

$$
\begin{array}{lll}
 \displaystyle r=1: \quad EW_{0:i:N} &   \displaystyle  =\frac{(N^2-1)(cN+1)}{6p}-\frac{(i^2-1)(ci+1)}{6p}, \\[16pt]
 \displaystyle r\neq 1: \quad EW_{0:i:N} & \displaystyle  = \frac{cN+1}{p(1-r)}\left( \frac{r+1}{r-1}-N\frac{r^N+1}{r^N-1}   \right) -\frac{ci+1}{p(1-r)}\left( \frac{r+1}{r-1}-i\frac{r^i+1}{r^i-1}   \right).  
\end{array}
$$
\end{corr}
\begin{proof} 
Applying Corollary  \ref{cor:cons_r_EW} we have:
\par \ 
\begin{itemize}
 \item $r=1$
$$ EW_{0:1:i} = \sum_{n=1}^{i-1} \frac{n/i}{p(n)} (i-n)=\frac{1}{p}\sum_{n=1}^{i-1} \frac{n}{i} (i-n)(2cn+1)=\frac{(i-1)(i+1)(ci+1)}{6p}.$$
\item $r\neq 1$
$$
\begin{array}{lllll}
 EW_{0:1:i} & = &  \displaystyle  \sum_{n=1}^{i-1}  \frac{ (1-r^n) (1-r^{i-n})}{p(n)(1-r^i)(1-r)}=   \frac{1}{p}\sum_{n=1}^{i-1}  \frac{ (1-r^n) (1-r^{i-n})}{(1-r^i)(1-r)}(2cn+1)\\[18pt]
 & = &  \displaystyle \frac{(ci+1)((r+1)(r^i-1)-i(r-1)(r^i+1)) }{p(1-r^i)(1-r)^2}\\[18pt]
 & =  &  \displaystyle \frac{ci+1}{p(1-r)}\left( \frac{r+1}{r-1}-i\frac{r^i+1}{r^i-1}   \right).
\end{array}
$$
\end{itemize}
%
Applying (\ref{eq:main_gambler_EWiN}), \ie $EW_{0:i:N} =EW_{0:1:N}-EW_{0:1:i}$, completes the proof.
\end{proof}


\subsection{The classical case.}\label{sec:gambl_class}
For constant winning/losing probabilities  we recover known results 
(all given in Sarkar \cite{Sarkar2006}). 
We state them here for completeness and will indicate how they can be derived 
from our more general results.

\begin{corr}
Consider the gambler's ruin problem on $\E=\{0,1,\ldots,N\}$ with constant winning/losing probabilities $p(i)=p, q(i)=q, i=1,\ldots,N-1, p+q=1$. We have
\begin{eqnarray*} 
 \rho_{0:i:N}& = & 
  \left\{ 
  \begin{array}{llll}
   {1-r^i\over 1-r^N}  & \rm{if } \ r=1,\\[10pt]
   {i\over N} & \rm{if } \ r\neq 1,   
  \end{array}
  \right.
\\[8pt]
 ET_{0:i:N} & = & 
  \left\{ 
  \begin{array}{llll}
   i(N-i) & \rm{if } \ r=1,\\[10pt]
   {r+1\over r-1}\left(i-N{r^i-1\over r^N-1}\right) & \rm{if } \ r\neq 1,   
  \end{array}
  \right.
\\[8pt]
 EW_{0:i:N} & = & 
  \left\{ 
  \begin{array}{llll}
   {1\over 3} (N-i)(N+i) & \rm{if } \ r=1,\\[10pt]
   {r+1\over r-1}\left[N {r^N+1\over r^N-1}-i{r^i+1\over r^i-1}\right] & \rm{if } \ r\neq 1,   
  \end{array}
  \right.
\\[8pt]
 EB_{0:i:N} & = & 
  \left\{ 
  \begin{array}{llll}
   {1\over 3} i(2N-i)&\rm{if } \  r=1,\\[10pt]
   {r+1\over r-1}\left[N {r^N+1\over r^N-1}-(N-i){r^{N-i}+1\over r^{N-i}-1}\right] & \rm{if } \ r\neq 1,   
  \end{array}
  \right.
\end{eqnarray*} 
\end{corr}

Results for $ET_{0:i:N}$ follows from Corollary \ref{cor:ET_constr} 
(case $r=1$);
$EW_{0:i:N}$ from Corollary \ref{cor:cons_r_EW} eq. (\ref{eq:EW01i_constr}) followed by 
(\ref{eq:main_gambler_EWiN}); 
$EB_{0:i:N}$ follows from results on $EW_{0:i:N}$ and Theorem
\ref{thm:main_cond} (eq. (\ref{eq:main_gambler_EB})).

 \subsection{Example}\label{sec:gam_example}
  Fix an integer N and some $p,q>0$. 
 Consider a gambler's ruin problem with  rates
$$p(i)={p(1+\alpha_1 i)\over 2ci+1}, \quad q(i)={q(1+\alpha_2 i)\over 2ci+1},$$
with fixed $\alpha_1, \alpha_2, c\geq 0$ such that $p(i),q(i)>0,p(i)+q(i)\leq 1, i\in \{1,...,N\}$.
We want to calculate $EW_{0:1:N}$.
%
\subsubsection{$N=3$}
We have 
$$\p =  \left(0,{p(1+\alpha_1)\over 2c+1}, {p(1+2\alpha_1)\over 2c+1},0\right), \quad \q=\left(0,{q(1+\alpha_2)\over 2c+1}, {q(1+2\alpha_2)\over 2c+1},0\right).$$
Note that in general 
(for $\alpha_1\neq \alpha_2$) $r(n)={q(n)\over p(n)}={q\over p} {(1+\alpha_2n)\over (1+\alpha_1n)}$ is
non-constant, thus we will apply Theorem \ref{thm:main_cond}.
Eq. (\ref{eq:main_gambler_EW1i}) takes form
$$ 
 EW_{0:1:3} = \displaystyle\sum_{n=1}^{2} \frac{\rho_{0:n:3}}{p(n)} \sum_{s=0}^{\lfloor (2-n)/2 \rfloor}\xi^{n+1,2}_s=\frac{\rho_{0:1:3}}{p(1)} 
\xi^{2,2}_0 + \frac{\rho_{0:2:3}}{p(2)} \xi^{3,2}_0.
$$
We need  winning probabilities $\rho_{0:1:3}$ and $\rho_{0:2:3}$, which can be calculated from
Theorem \ref{thm:main_rho_ET}:
$$\rho_{0:i:3}={\displaystyle  \sum_{n=1}^i \prod_{s=1}^{n-1}r(s) \over 
 \displaystyle   \sum_{n=1}^3 \prod_{s=1}^{n-1} r(s)}=\frac{1+(i-1)r(1)}{1+r(1)+r(1)r(2)} 
 ={1 + (i-1){q\over p} {1+\alpha_2\over 1+\alpha_1} \over 1+{q\over p} {1+\alpha_2\over 1+\alpha_1}+{q^2\over p^2} {(1+\alpha_2)(1+2\alpha_2)\over (1+\alpha_1)(1+2\alpha_1)}}
 =:{1 + (i-1){q\over p} {1+\alpha_2\over 1+\alpha_1} \over \gamma(p,q,\alpha_1,\alpha_2)}. $$
We also need $\xi^{2,2}_0$ and $\xi^{3,2}_0$. We have $\mathbf{j}_0^{2,2}=\mathbf{j}_0^{3,2}=\{\emptyset\}$, thus 
$$\xi_0^{2,2}= \delta_\mathbf{j}^{2,2}=1+r(2)=1+{q\over p} {1+2\alpha_2\over 1+2\alpha_1},\quad \xi_0^{3,2}= \delta_\mathbf{j}^{3,2}=1$$
(in the latter the second product was 1, since $\{3,\ldots,2\}\equiv\emptyset$).
 
Finally,
\begin{equation}\label{ex:ew13}
  EW_{0:1:3} = {1\over p \gamma(p,q,\alpha_1,\alpha_2)} 
 \left[{2c+1\over  1+\alpha_1}\left(1+{q\over p} {(1+2\alpha_2)\over(1+2\alpha_1)}\right) + \left(1+{q\over p} {(1+\alpha_2)\over (1+\alpha_1)}\right){4c+1\over 1+2\alpha_1}\right].
\end{equation}

 Special cases:
\begin{itemize}
 \item $\alpha_1=\alpha_2=\alpha$. Then (\ref{ex:ew13}) reduces to
 \begin{equation}\label{ex:ew13_conr}
  EW_{0:1:3} = 
  {1+{q\over p} \over p\left(1+{q\over p}+{q^2\over p^2}\right)}
  \left( {2c+1\over 1+\alpha} + {4c+1\over 1+2\alpha}\right).
 \end{equation}

  \textsl{Note}  that in this case $r(n)={q\over p}$ is constant, thus 
  (\ref{ex:ew13_conr}) could be derived  in an easier way using Corollary 
  \ref{cor:cons_r_EW}:
  $$
  \begin{array}{llll}   
  r = 1: &  EW_{0:1:3} & =\displaystyle\sum_{n=1}^{2} {n\over 3} {2cn+1\over p(1+\alpha_1n)} (3-n)
  ={2\over 3p}\left( {2c+1\over 1+\alpha}+{4c+1\over 1+2\alpha}\right),
   \\[20pt]
  r\neq 1: &  EW_{0:1:3} & =\displaystyle\sum_{n=1}^{2} { {1-r^n\over 1-r^3}(1-r^{3-n}) \over (1-r)}{2cn+1\over p(1+\alpha_1n)}
  = {1-r^2\over p(1-r^3)}\left( {2c+1\over 1+\alpha}+{4c+1\over 1+2\alpha}\right),  \\[20pt]
  \end{array}
$$   
what is equivalent to (\ref{ex:ew13_conr}) in both cases. \textsl{Note} also  that this is not a spatially non-homogeneous  case as long as $\alpha>0$.
 
\item $\alpha_1=\alpha_2=0$. Then (\ref{ex:ew13}) (and thus (\ref{ex:ew13_conr})) reduces to
 \begin{equation}\label{ex:ew13_conr_spat}
  EW_{0:1:3} = 
  {2\left(1+{q\over p}\right) \over p\left(1+{q\over p}+{q^2\over p^2}\right)}
  \left( {3c+1\over 1+\alpha}  \right).  
 \end{equation}
 Note that this is a spatially non-homogeneous  case, thus (\ref{ex:ew13_conr_spat}) could be derived from Corollary \ref{cor:spat_hom} 
 (we skip the calculations).
 
 \item $\alpha_1=\alpha_2=0$ and $c=0$, then (\ref{ex:ew13_conr_spat}) reduces to 
  $$
  EW_{0:1:3} = 
  {2\left(1+{q\over p}\right) \over p\left(1+{q\over p}+{q^2\over p^2}\right)}.
  $$
  This situation corresponds to a gambler's ruin problem with constant birth and death rates. In particular, for $p=q=1/2$ we have 
  $EW_{0:1:3}={8\over 3}$ what agrees with Example 1 in \cite{Lengyel2009a}.\par

\end{itemize}
  
 \subsubsection{General $N\geq 3, p=q$ and $\alpha_2=\alpha_1=1$}
 We thus have $p(i)={p(1+i)\over 2ci+1}, q(i)={q(1+ i)\over 2ci+1}$. This is
 constant $r(n)={q(n)\over p(n)}={q\over p}=1$ case, which is 
 however not spatially non-homogeneous. 
 We skip the lengthy calculations, but we can obtain $EW_{0:1:N}$ from 
 Corollary \ref{cor:cons_r_EW} ($H_N$ is the $N$-th harmonic number):
 $$
 \begin{array}{lll}
  EW_{0:1:N} & = & \displaystyle \sum_{n=1}^{N-1}{n(N-n)(2cn+1)\over p N(1+n)} \\[16pt]
  
  & = & \frac{1}{p}\big(\frac{c}{3}(N-5)(N+2)+\frac{1}{2}(3+N)\big)+\frac{1}{Np}(2c-1)(1+N)H_N = {c\over 3p} N^2 + \Theta(N),\\

 \end{array}
 $$
 which for $p(i)=p(1+i), q(i)=q(1+i)$ (\ie for $c=0$) simplifies to 
 \par 
 $$ EW_{0:1:N} = \displaystyle {N+3\over 2p} -{1\over Np}(N+1)H_N =  {N \over 2p} + \Theta(log(N)) .$$

   \subsubsection{General $N\geq 3, p=q$ and $\alpha_2=\alpha_1=\alpha$}
    $$
 \begin{array}{lll}
  EW_{0:1:i} & = & \displaystyle \sum_{n=1}^{i-1}{n(i-n)(2cn+1)\over p i(1+\alpha n)} \\[16pt]
  
  & = & \displaystyle {1\over 6\alpha^4 p i} 
  \left( \alpha(i-1)(\alpha^2 i (2c(i+1)+3)+\alpha(6-6ci)-12c)+ \right. \\[16pt]
  & & \left. 6(\alpha-2c)(\alpha i+1)\left[ \psi\left(1+{1\over \alpha}\right) - \psi\left(i+{1\over \alpha}\right)\right]\right),
 \end{array}
 $$
 where $\psi$ is a digamma function. It is known that $\psi(m)=H_{m-1}-\gamma$, where $\gamma=0.5772156...$ is a 
 known  Euler–Mascheroni constant.
 Let us assume that $\alpha={1\over m}$ and $m$ is an integer. Then 
 $\psi\left(1+{1\over \alpha}\right) - \psi\left(i+{1\over \alpha}\right)=H_{m}-H_{i+m-1}$.
 
%
%
%
%
%
%
%
%
 
 \color{black}
 
 \section{Random walk on a polygon}\label{sec:polygon}

Fix an integer $m\geq 2$. Let 
 $$\p=(p(0), p(1),\ldots, p(m)), \quad \q=(q(0), q(1),\ldots, q(m)),$$
 where $p(i),q(i)>0, p(i)+q(i)\leq 1$ for $i \in \{0\ldots,m\}$.
 Consider the following random walk $\X\equiv \{X_t\}_{t\in\mathbb{N}}$ on $\E=Z_{m+1}$. 
 Being in state $i$ we move to the state $i+1$ with probability $p(i)$, we move to the state $i-1$ with probability $q(i)$,
 and we do nothing with the remaining probability.  
 Throughout the paper, in the 
 context of a random walk on a polygon,
 all additions and substractions are performed modulo $m+1$.
 We will refer to this walk as to a \textsl{random walk on a polygon}.  The notation intentionally resembles that of gambler's ruin problem.
 Throughout the section we consider fixed $\p, \q$ and $m\geq 2$ (and omit subscripts $\p,\q$ in random variables below).
      We are interested in:

 \begin{eqnarray*}
 A_i & = & \{X: X_0=i,X_n=i, \forall_{0<t<n}X_t\neq i,\forall_{k \in \E}\exists_{0\leq t\leq n} X_t=k \} \\[8pt]
 L_{i,j} & = & \{X: X_0=i,X_n=j, \forall_{0<t<n}X_t\neq j,\forall_{k \in \E}\exists_{0\leq t\leq n} X_t=k \} \\[8pt]
 V_{i,j} & = & \inf \{n\geq 1: X_0=i , X_n=j, \forall_{k \in \E}\exists_{0\leq t\leq n} X_t=k \} \\[8pt]
 V_i & = & \inf \{n\geq 1: X_0=i , \forall_{k \in \E}\exists_{0\leq t\leq n} X_t=k \} \\[8pt]
 R_i & = &  \inf \{n_2\geq 1:  X_0=i , X_{n_1+n_2}=i, n_1=\inf\{n\geq 1: \forall_{k \in \E}\exists_{0\leq t\leq n} X_t=k \}\} \\[8pt]
\end{eqnarray*}
In other words: $A_i$ is the event that the process starting at $i$ will return   for the first 
time to $i$ after all other vertices are visited; $L_{i,j}$ is the event that the process
starting at $i$ will reach for the first time state $j$ after visiting all other vertices;
$V_{i,j}$ is the number of steps of the process starting at $i$ to reach for the 
first time state $j$ after visiting all other vertices; $V_i$ is the number 
of steps of the process starting at $i$ needed to visit all vertices;
$R_i$ is the number of additional steps for the process starting at $i$
needed to reach $i$ after visiting all the vertices.

\medskip\par 

For $j\preceq i \preceq k$, where $\preceq$ is a cyclic order, \ie $j\leq i\leq k$ or $i\leq k \leq j$ 
or $k \leq j \leq i$, let $G(\p,\q,j,i,k)$
denote a gambler's ruin game with $i$ being a starting state, $j$ being a losing state and 
$k$ being a winning state. Note that independently of $j,i,k$,  winning and losing 
probabilities $\p, \q$ are fixed.

\textbf{Notation.}  
\label{polygon:notation}
In contrast to a usual notation neither $\sum_{k=s}^t a_k=0$ nor $\prod_{k=s}^t a_k=1$ 
for $t<s-1$. Since we are considering operations in $Z_{m+1}$,
we define 

$$\begin{array}{lllll}
  \textrm{For }t<s\leq m, s-t>1: & \displaystyle \sum_{k=s}^t a_k := a_s + a_{s+1} + \ldots + a_m + a_0 +\ldots +a_t,\\[20pt]
    & \displaystyle \prod_{k=s}^t a_k := a_s \cdot a_{s+1} \cdot \ldots \cdot a_m \cdot a_0 \cdot \ldots \cdot a_t,  \\[20pt]

    \textrm{For }s=t+1 \bmod m+1: & \displaystyle \sum_{k=s}^t a_k = 0 \qquad \prod_{k=s}^t a_k := 1.  \\[20pt]    
  \end{array}
$$
In all other cases we use usual sums and products.
%
Using this notation, we are ready to state our results.
 \begin{theorem}\label{thm:main_polygon}
Consider the random walk on a polygon described above. We have
%

\begin{eqnarray}
P(A_i) & = & {1\over \displaystyle 1+r(i)}\left({ 1\over \displaystyle\sum_{n=i+1}^{i-1} \prod_{s=i+1}^{n-1} r(s)}+ {1\over \displaystyle \sum_{n=i+2}^i \prod_{s=n}^{i} \left({1\over r(s)}\right)}\right) \label{thm:main_poly_Ai} \\[8pt]
P(L_{i,j})&=& 
  {\displaystyle  1 \over \displaystyle   \sum_{n=j+2}^{j-1} \prod_{s=j+2}^{n-1}r(s)} \left(
  {\displaystyle  \sum_{n=i+1}^{j-1} \prod_{s=j+2}^{n-1}r(s) \over 
 \displaystyle   \sum_{n=j+1}^{j-1} \prod_{s=j+1}^{n-1}r(s)} +
 {\displaystyle  \sum_{n=j+2}^{i} \prod_{s=j+2}^{n-1}r(s) \over 
 \displaystyle   \sum_{n=j+2}^{j} \prod_{s=n}^{j-1}\frac{1}{r(s)}} \right) \label{thm:main_poly_Lij} \\[5pt]
 EV_{i,j}&=&\rho_{j+1:i:j-1} \left(EW_{j+1:i:j-1}  +EB_{j+1:j-1:j}+ ET_{j:j+1:j}  \right)\nonumber\\
 & & + (1-\rho_{j+1:i:j-1} ) \left(EB_{j+1:i:j-1} +EW_{j:j+1:j-1}+ET_{j:j-1:j}\right) \label{thm:main_poly_Vij}\\
 EV_i & = & \sum_{j=i+1}^{i-1}P(L_{i,j})EV_{i,j}\label{thm:main_poly_Vi}\\ 
ER_{i}&=&\sum_{k=i+1}^{i-1}P(L_{i,k})ET_{i:k:i}\label{thm:main_poly_Ri}
\end{eqnarray} 

\end{theorem}

The proof of Theorem \ref{thm:main_polygon}  is  postponed to Section \ref{sec:proof_polygon}.

%

\paragraph{Constant $r(n)=r={q(n)\over p(n)}$.} \ \smallskip
\par
In this case the starting point does not matter, we consider $i=0$.
Note that $P(A_i)$ and $P(L_{i,j})$ depend on $p(n)$ and $q(n)$ only through $r(n)$, thus they must reduce 
to known results for constant birth $p(n)=p$ and death $q(n)=q$ rates (see (3.1) and (3.3) in \cite{Sarkar2006}).
Indeed, substituting $r(n)=r$ to (\ref{thm:main_poly_Ai}) and (\ref{thm:main_poly_Lij}) yields
\begin{corr}
 Consider the random walk on polygon with constant $r(n)={q(n)\over p(n)}$, then we have 
 \begin{eqnarray*}
P(A_0) & = & 
\left\{ 
\begin{array}{llll}
{1\over m} & \textrm{if } \ r=1, \\
{r-1\over r+1} {r^m+1\over r^m-1} & \textrm{if } \ r\neq 1,
\end{array}
\right. 
\label{thm:main_poly_const_Ai} \\[8pt]
P(L_{0,j}) & = & 
\left\{ 
\begin{array}{llll}
{1\over m} & \textrm{if } \ r=1,\\
{r^{m-j}(r-1)\over r^m-1}  & \textrm{if } \ r\neq 1.
\end{array}
\right. 
\end{eqnarray*}
\end{corr}
We skip the formulas for $EV_{0,j}, EV_0$ and $ER_0$ in this case, noting that they can be derived 
from Corollaries \ref{cor:ET_constr} and  \ref{cor:cons_r_EW}.

\paragraph{Constant $q(n)=q,p(n)=p$} \ \smallskip
\par
\noindent
First, let us recall formulas for $EV_0$, $ER_0$ for the case $p+q=1$.

\begin{corr}{\cite{Sarkar2006}}
 Consider the random walk on a polygon with constant $q(n)=q,p(n)=p, p+q=1$. We have

\begin{eqnarray*} 
EV_0 & = & 
\left\{ 
\begin{array}{llll}
{m(m+1)\over 2} & \textrm{if } \ r=1,\\
{r+1\over r-1} \left[m-{1\over r-1} - {m^2\over r^m-1} + {(m+1)^2\over r^{m+1}-1}\right]  & \textrm{if } \ r\neq 1,
\end{array}
\right. 
\label{thm:main_poly_const_Vi} \\[8pt]
ER_0 & = & 
\left\{ 
\begin{array}{llll}
{1\over 6}(m+1)(m+2) &  \textrm{if } \ r=1,\\
{r+1\over r-1} \left[{r\over r-1} - {m(m+2)\over r^m-1} + {(m+1)^2\over r^{m+1}-1}\right]  & \textrm{if } \ r\neq 1,
\end{array}
\right. 
\end{eqnarray*} 
\end{corr}

In the case $p+q\leq 1$  note that $EB_{j:i:k}=\frac{1}{p(1+r)}EB^1_{j:i:k}$, $EW_{j:i:k}=\frac{1}{p(1+r)}EW^1_{j:i:k}$, $ET_{j:i:k}=\frac{1}{p(1+r)}ET^1_{j:i:k}$, 
where superscript $1$ denotes the case $p+q=1$. Thus Theorem \ref{thm:main_polygon} implies $EV_0=EV_0^1, ER_0=ER_0^1$, \ie we have

\begin{corr}
 Consider the random walk on a polygon with constant $q(n)=q,p(n)=p$. We have

\begin{eqnarray*} 
EV_0 & = & 
\left\{ 
\begin{array}{llll}
{m(m+1)\over 4p} & \textrm{if } \ r=1,\\
{1\over p(r-1)} \left[m-{1\over r-1} - {m^2\over r^m-1} + {(m+1)^2\over r^{m+1}-1}\right]  & \textrm{if } \ r\neq 1,
\end{array}
\right. 
\label{thm:main_poly_const_Vi} \\[8pt]
ER_0 & = & 
\left\{ 
\begin{array}{llll}
{1\over 12p}(m+1)(m+2) &  \textrm{if } \ r=1,\\
{1\over p(r-1)} \left[{r\over r-1} - {m(m+2)\over r^m-1} + {(m+1)^2\over r^{m+1}-1}\right]  & \textrm{if } \ r\neq 1,
\end{array}
\right. 
\end{eqnarray*} 
\end{corr}

  \section{Fastest Strong Stationary Time for a symmetric random walk on   a circle}\label{sec:polygon} 
 
\tikzstyle{block3} = [draw,fill=gray!50,minimum size=0.5em]
Consider an ergodic Markov chain $\X=\{X_k\}_{k\geq 0}\sim(\nu,\PP_X)$ on a 
finite state space $\E=\{\e_1,\ldots,\e_M\}$ with a 
stationary distribution $\pi$, initial distribution $\nu$ 
and a transition matrix $\PP_X$. We are interested in
measuring  nonstationarity of $\X_k$ via \textbf{separation ``distance''}
$$sep(\nu\PP^k_X,\pi)=\max_{\e\in\E}\left(1-{\nu\PP_X^k(\e)\over \pi(\e)}\right).$$
Note that it is not symmetric, that is why it is not an actual distance,
however it is an upper bound on a \textbf{total variation distance} $d_{TV}(\nu\PP_X^k,\pi)=1/2 
\sum_{\e\in\E}|P(X_k=\e)-\pi(\e)|$.
\par 
\noindent 
A random variable $T$ is a \textbf{strong stationary time} (SST) $T$ for $\X$ if 
it is a randomized stopping time for $\X$ such that 
$$\forall(\e\in\E) \ P(X_k=\e | T=k) = \pi(\e).$$
The notion of separation distance fits perfectly into a notion of SST,\
in \cite{Aldous1987}  it is shown that for an SST $T$ we 
have 
\begin{equation*}\label{eq:TV_sep_SST}
sep(\nu\PP_X^k,\pi)\leq P(T>k).
\end{equation*}
We say that $T$ is a \textbf{fastest strong stationary time} (FSST) if 
$sep(\nu\PP_X^k,\pi)=P(T>k).$
\medskip\par 

In this section we consider a symmetric random walk on a polygon with constant 
rates $p(i)=q(i)=p $ on $d$ points (\ie $m=d-1$).
Moreover, we will refer to the random walk as to a 
\textsl{symmetric random walk on a circle} (to be consistent with \cite{Diaconis1990a},
we will compare our result to a result from this article) on
$\mathbb{Z}_d, $ \ie  $\{0,\ldots,d-1\}$.
We will show a construction of a fastest strong stationary time for this symmetric 
random walk on a circle, moreover we have 
\begin{lemma}\label{lem:fsst_circle}
 For the fastest strong stationary time $T$ for a symmetric random walk on 
 a circle with 
 $d=2N$ we have 
 $$ET = \left\{ 
 \begin{array}{lll}
 \displaystyle {2N^2+1\over 12p} & \textrm{for } p\in(0,1/3] \textrm{ and } N>1, \\[12pt]
 \displaystyle {1\over 4p} & \textrm{for } p\in(0,1/4] \textrm{ and } N=1, \\[12pt]
 \displaystyle {1\over 2(1-2p)} & \textrm{for } p\in(1/4,1/2) \textrm{ and } N=1. \\
 \end{array}\right.
$$
\end{lemma}

\begin{remark}
A construction of a strong stationary time  for a 
symmetric random walk on a circle with $p=1/3$    is presented 
in \cite{Diaconis1990a}. For $d=2^a, a>1$ their construction yields an SST $T_0$ such that 
$$ET_0={3\over 2} 2^{2a}\left(2^{-4}+2^{-6}+\cdots + 2^{-2(a-1)} + 
2\times 2^{-2a}\right) = {1\over 8} d^2+1$$
(see the bottom of the page 1484 in \cite{Diaconis1990a}),
whereas Lemma \ref{lem:fsst_circle} states that a fastest strong stationary time $T$ 
fulfills ($N=d/2$)
$$ET={1\over 8} d^2+{1\over 4},$$
what means that a construction from  \cite{Diaconis1990a} does not yield a 
\textsl{fastest} strong stationary time (authors mention this fact in their Example 3.1).
Note that $ET$ and $ET_0$ differ by  ${3\over 4}$ (independently of $d$).
\end{remark}

\paragraph{Strong stationary duality} For a general description of a strong stationary duality 
see \cite{Diaconis1990a} (total ordering and set-valued chains) and 
\cite{Lorek2012d}, \cite{Lorek2016_Siegmund_duality} (general partial ordering). 
Here we will describe this duality for chains on the same state space.
Let both $\X \sim(\nu,\PP_X)$ and $\X^*  \sim(\nu^*,\PP_X^*)$ be chains on 
$\E=\{\e_1,\ldots,\e_M\}$, chain $\X$ is ergodic with a stationary distribution $\pi$,
whereas $\X^*$ is an absorbing chain with a unique absorbing state $\e_M$.
We say that a stochastic matrix of size $d\times d$ is a \textbf{link} if 
$\Lambda(\e_M,\e)=\pi(\e)$ for all $\e\in\E$. We say that $\X^*$ is a 
\textbf{strong stationary dual} of $\X$ with the link $\Lambda$ if 
\begin{equation}\label{eq:ssd_duality}
\nu=\nu^*\Lambda \ \textrm{ and } \ \Lambda\PP_X=\PP_X^*\Lambda.
\end{equation}
Diaconis and Fill \cite{Diaconis1990a} proved that the absorption time 
$T^*$ of $\X^*$ is an SST for $\X$. If the corresponding $T^*$ is an FSST for $\X$, 
then the chain $\X^*$ is called a \textbf{sharp SSD}.
\par 
Fix some partial ordering $\preceq$ on $\E$, such that $\e_1$ is the minimum and 
$\e_M$ is the maximum. Let $\C(\e_i,\e_j)=\mathbf{1}(\e_i\preceq \e_j)$ be the 
corresponding \textsl{ordering matrix} (always invertible,
the inverse $\C^{-1}$ is called the M\"obius matrix).
Assume that $\nu(\e_1)=1$ (\ie chain $\X$ starts in $\e_1$), then (\ref{eq:ssd_duality})
implies that also  $\nu^*(\e_1)=1$. Let $\overleftarrow{\PP}_X$ be a transition matrix 
of a time reversed chain, \ie $\overleftarrow{\PP}_X(\e_i,\e_j)={\pi(\e_j)\over \pi(\e_i)} 
\PP_X(\e_j,\e_i)$.
We have 
\begin{theorem}[Theorem 2 in \cite{Lorek2012d}, Remark 2.2 in \cite{LorekSzekli2016}, simplified version]
\label{thm:ssd}
Let
 $\X  \sim(\nu,\PP_X)$ be an ergodic Markov chain on a finite 
 state space $\E=\{\e_1,\ldots,\e_M\}$ starting at $\e_1$ (\ie $\nu(\e_1)=1$), with a stationary 
 distribution $\pi$, partially ordered by $\preceq$ (with
 ordering matrix $\C$),
 with $\e_1$ being the minimum and $\e_M$ being the maximum.
 Assume that $\C^{-1}\overleftarrow{\PP}_X\C$ is a non-negative matrix.
 Then there exists a \textsl{sharp SSD} $\X^*\sim(\nu^*,\PP^*_X)$ on $\E$ with $\nu^*(\e_1)=1$ and 
 transitions
 \begin{equation}\label{eq:dual_trans}
 \PP_X^*(\e_i,\e_j) = {H(\e_j)\over H(\e_i)} \left(\C^{-1}\overleftarrow{\PP}_X\C\right)(\e_j,\e_i)
 \end{equation}
 with a unique absorbing state $\e_M$, where $H(\e)=\sum_{\e'\preceq \e} \pi(\e)$.
\end{theorem}
\begin{remark}
 The condition that $\C^{-1}\overleftarrow{\PP}_X\C$ is a non-negative matrix 
 was called ${}^\downarrow$-M\"obius monotonicity in \cite{Lorek2012d}.
\end{remark}

\begin{proof}[Proof of Lemma \ref{lem:fsst_circle}]
 
First, we will construct a sharp SSD for this symmetric random walk on a circle using Theorem \ref{thm:ssd}. It will be more convienient to work with states numerated as $1^*,2^*,\ldots, d^*=(2N)^*$ (instead of $0,1,\ldots,2N-1$).
Our walk $\X$ moves either right or left, or it does not move, 
i.e., it has the  transition matrix:
$$
\PP_X(i^*,j^*)=
\left\{ 
\begin{array}{llll}
 1-2p & \textrm{if} & j^*=i^*, \\[8pt]
 p& \textrm{if} &  (j^*=(i+1)^*, i^*\neq (2N)^*)  \ \lor \ (j^*=(i-1)^*, i^*\neq 1^*) \\[8pt]
 &  & \lor \ (j^*=1, i^*=(2N)^*) \lor (j^*=(2N)^*, i^*=1).\\ 
\end{array}
\right.
$$
It will be even more convienient to work with another enumeration of states.
Consider a set of states $\{1,\ldots,2N\}$ and let us define a bijection between
this set and the set $\{1^*,\ldots,(2N)^*\}$ in the following way:

$$
  \sigma(i^*) = 
  \left\{ 
  \begin{array}{lllll}
  2i-1 & \textrm{ if } & i\leq N, \\[10pt]
  2(2N-i+1) & \textrm{ if } & i>N.  
 \end{array}
 \right., \qquad
  \sigma^{-1}(i) = 
  \left\{ 
  \begin{array}{lllll}
  \left({i+1\over 2}\right)^* & \textrm{ if } &  \textrm{ is odd}, \\[10pt]
  \left( 2N-{i\over 2}+1\right)^* & \textrm{ if } &  \textrm{ is even}.
 \end{array}
 \right. 
$$
The bijection for  $d=2N=8$ is following 
\begin{eqnarray*}
 \sigma((1^*,2^*, 3^*, 4^*, 5^*, 6^*, 7^*, 8^*))&=&(1,3,5,7,8,6,4,2), \\[6pt]
 \sigma^{-1}((1,2, 3, 4, 5, 6, 7, 8)) &=&(1^*,8^*,2^*,7^*,3^*,6^*,4^*,5^*),
\end{eqnarray*}
it  is depicted in Fig. \ref{fig:circle_zig} (left).
The transition matrix of the chain $\X$ can be rewritten as:
$$
\PP_X(i,j)=
\left\{ 
\begin{array}{llll}
 1-2p & \textrm{if} & i=j, \\[8pt]
 p& \textrm{if} &  |i-j|=2 \ \lor \ (i=1, j=2) \ \lor \ (i=2, j=1)\  \lor \\[6pt]
 & & (i=2N-1, j=2N) \  \lor \ (i=2N, j=2N-1). \\
\end{array}
\right.
$$

 \ \par\noindent
Continuing our example $d=2N=8$ we have (using enumeration of states $1,2,\ldots,2N)$)
$$ 
{\small 
\PP_X=\left[ 
\begin{array}{cccccccc} 
1-2p&p&p&0&0&0&0&0\\[2pt]
p&1-2p&0&p&0&0&0&0\\[2pt]
p&0&1-2p&0&p&0&0&0\\[2pt]
0&p&0&1-2p&0&p&0&0\\[2pt]
0&0&p&0&1-2p&0&p&0\\[2pt]
0&0&0&p&0&1-2p&0&p\\[2pt]
0&0&0&0&p&0&1-2p&p\\[2pt]
0&0&0&0&0&p&p&1-2p
\end{array} 
\right]
}
$$
We will now compute an SSD chain using total ordering $1<2<\ldots<2N$.
Mapping the total ordering $1<2<\ldots 2N$ into the ordering on original states $1^*,2^*,\ldots, (2N)^*$, we have 
$$
i^*\prec j^* \Leftrightarrow \sigma(i^*)<\sigma(j^*),
$$
Note that $\prec$ is also a total ordering, we have 
$$1^*\prec (2N)^* \prec 2^* \prec (2N-1)^* \prec \cdots \prec (N+1)^*.$$
We will thus work with total ordering $1<2\ldots<2N$ -- which is equivalent (with easier notation) to working with $1^*\prec (2N)^*\prec\ldots\prec (N+1)^*$.
%
%
%
%
%
%
%
%
%
%
%
%
%
\ \par\noindent 
The ordering matrix for total ordering is $\C(i,j)=\mathbf{1}(i\leq j)$,
the  M\"obius matrix (\ie the inverse of $\C$) 
is then following:
$$
\C^{-1}(i,j)=
\left\{ 
\begin{array}{llll}
 1 & \textrm{if} & i=j, \\[3pt]
 -1& \textrm{if} & i=j-1, i< 2N. \\
\end{array}
\right.
$$
%
%
 \noindent
We have
\begin{equation}\label{eq:circle_Hi}
H(i)=\sum_{j\leq i} \pi(j) =\sum_{j\leq i} {1\over 2N}
={ i\over 2N}.
\end{equation}
 
 \par
  \noindent
 Using above derivations and the fact that the chain is reversible 
($\overleftarrow{\PP}_X=\PP_X)$, for $i<2N$ we have:

$$
\begin{array}{lllll}
\left(\C^{-1}\overleftarrow{\PP}_X\C\right)(i,j) & = & \displaystyle \left(\C^{-1}\PP_X\C\right)(i,j) = \sum_{l} \C^{-1}(i,l)\sum_{k\leq  j} \PP_X(l,k) \\[10pt]
 & = & \displaystyle 
   \sum_{k\leq  j} \PP_X(i,k)-\PP_X(i+1,k) \\ 
 & = & \displaystyle\PP_X(i,j)+ \left( \sum_{k< j} \PP_X(i,k)- \PP_X(i+1,k+1)\right) -\PP_X(i+1,1),
\end{array}
$$
whereas for $i=2N$ we have
$$ 
\left(\C^{-1}\overleftarrow{\PP}_X\C\right)(i,j)= \sum_{k\leq j} \PP_X(2N,k)=
\left\{ 
\begin{array}{llll}
 0 & \textrm{if} & j< 2N-2, \\[3pt]
 p & \textrm{if} & j= 2N-2, \\[3pt]
 2p & \textrm{if} & j= 2N-1, \\[3pt]
 1& \textrm{if} & j= 2N, \\
\end{array}
\right.
$$

\noindent 
We also have:
$$\PP_X(i,k)- \PP_X(i+1,k+1) =
\left\{ 
\begin{array}{rlll}
 p & \textrm{if} & (i=1, j=2)\  \lor\  (i=2, j=1), \\[8pt]
 -p& \textrm{if} & (i=2N-1, j=2N-2)\  \lor\ (i=2N-2, j=2N-1). \\
\end{array}
\right.
$$
\noindent 
Using above derivations we can easily calculate all the cases:

\begin{equation}\label{eq:circle_CinvPC}
\left(\C^{-1}\overleftarrow{\PP}_X\C\right)(i,j) =
\left\{ 
\begin{array}{llll}
 1-2p & \textrm{if} & 1< i=j < 2N-1, \\[6pt]
 1-3p & \textrm{if} & i=j=1 \ \lor \  i=j=2N-1, \\[6pt]
 1 & \textrm{if} & i=j=2N, \\[6pt]
 p & \textrm{if} & |i-j|=2,\  j \neq 2N, \\[6pt]
 2p& \textrm{if} & i=2N,\ j=2N-1. \\
\end{array}
\right.
\end{equation}

 \ \par\noindent
Continuing  our example  $d=2N=8$ we have (again, using enumeration $1,2,\ldots, 8$)  
$$ 
{\small 
\C^{-1}\overleftarrow{\PP}_X\C
=\left[ 
\begin{array}{cccccccc} 
1-3p&0&p&0&0&0&0&0\\[2pt]
0&1-2p&0&p&0&0&0&0\\[2pt]
p&0&1-2p&0&p&0&0&0\\[2pt]
0&p&0&1-2p&0&p&0&0\\[2pt]
0&0&p&0&1-2p&0&p&0\\[2pt]
0&0&0&p&0&1-2p&0&0\\[2pt]
0&0&0&0&p&0&1-3p&0\\[2pt]
0&0&0&0&0&p&2p&1
\end{array} 
\right].
}
$$
Combining (\ref{eq:circle_CinvPC}) with (\ref{eq:circle_Hi}) and noting 
that ${H(i) \over H(j)} = {i \over j} $,
 Theorem \ref{thm:ssd} and   yields the following transitions 
  of a sharp SSD chain $\X^*$ (written using the original enumeration of states)
 $$ 
\PP_X^*(i^*,j^*)=
{H(j^*) \over H(i^*)}\left(\C^{-1}\overleftarrow{\PP}_X\C\right)(j^*,i^*) =
{H(\sigma(j^*)) \over H(\sigma(i^*))}\left(\C^{-1}\overleftarrow{\PP}_X\C\right)(\sigma(j^*),\sigma(i^*)), 
$$ 
thus  
$$
\PP_X^*(i^*,j^*)=
\left\{ 
\begin{array}{llll}
 1-2p & \textrm{if} & 1< \sigma(j^*)=\sigma(i^*) < 2N-1, \\[6pt]
 1-3p & \textrm{if} & \sigma(j^*)=\sigma(i^*)=1 \lor \sigma(j^*)=\sigma(i^*)=2N-1, \\[6pt]
 1 & \textrm{if} & \sigma(j^*)=\sigma(i^*)=2N, \\[6pt]
 p{\sigma(j^*) \over \sigma(i^*)} & \textrm{if} & |\sigma(j^*)-\sigma(i^*)|=2, \sigma(i^*) \neq 2N, \\[6pt]
 2p{\sigma(j^*) \over \sigma(i^*)} & \textrm{if} & \sigma(j^*)=2N,\sigma(i^*)=2N-1. \\
\end{array}
\right.
$$
We leave it to the reader to check that the condition $|\sigma(j^*)-\sigma(i^*)|=2$ for $j,i\leq N$ or $ j,i> N$ is 
equivalent to $|j-i|=1$, whereas for $j\leq N, i>N$ or for $i\leq N, j>N$
the condition is never met.
Thus, the transition matrix of $\X^*$ can rewritten in the following way, using ordering $\prec$:
$$
\PP^*_X(i^*,j^*)
=
\left\{
\begin{array}{llll}
1-2p & \textrm{if } & j^*=i^*, 2^*\preceq i^* \prec N \textrm { or } (N+1)^*\prec i^*\preceq (2N)^*,  \\[7pt]
1-3p & \textrm{if } & j^*=i^*, i^*\in\{1^*,N^*\}, \\[7pt]
1 & \textrm{if } & j^*=i^*=(N+1)^*, \\[7pt]

{(2i+1)p\over 2i-1} & \textrm{if } & j=i+1, 1^*\preceq i^* \prec N^*, \\[7pt]
{(2i-3)p\over 2i-1} & \textrm{if } & j=i-1, 1^* \prec i^* \preceq N^*, \\[7pt]

{(2N-i)p\over 2N-i+1} & \textrm{if } & j=i+1,  (N+2)^* \preceq i^*\prec (2N)^*,\\[7pt]
{(2N-i+2)p\over 2N-i+1} & \textrm{if } & j=i-1,  (N+2)^* \preceq i^*\preceq (2N)^*, \\[7pt]

{4Np\over 2N-1} & \textrm{if } & i^*=N^*, j^*=(N+1)^*.
\end{array}
\right.
$$

\par

\begin{itemize}
\item 
 First, let us consider case $p\in(0,1/3]$ and $N>1$. \par
Note that the assumption $p\in(0,1/3]$ implies that $\PP_X^*$ is a transition matrix. Continuing the example   $d=2N=8$, we have (using the enumeration $1^*,2^*,\ldots, (2N)^*)$
$$
\footnotesize{
\PP^*_X=
\left[
\begin{array}{ccccccccccc}
1-3\,p&3\,p&0&0&0&0&0&0\\ 
p/3&1-2\,p&5/3\,p&0&0&0&0&0\\ 
0& 3/5\,p&1-2\,p&7/5\,p&0&0 & 0 & 0\\ 
0&0&5/7\,p&1-3\,p&{\frac {16\,p}{7}}&0&0&0\\
0&0&0&0&1&0&0&0\\
0&0&0&0&4/3\,p&1-2\,p&2/3\,p&0\\
0&0&0&0&0&3/2\,p&1-2\,p&p/2\\
0&0&0&0&0&0&2\,p&1-2\,p
\end {array} 
\right].
}
$$

\smallskip\par
Note that the resulting chain (recall, it  starts at $1^*$) will never reach states
$(N+2)^*,\ldots,(2N)^*$. Denote the resulting chain on $\{1^*,\ldots, (N+1)^*\}$ by $\Y^*$.
This is a birth and death chain with a unique absorbing state $(N+1)^*$, let us write down the relevant 
transitions only 

$$
\PP^*_Y(i^*,j^*)
=
\left\{ 
\begin{array}{llll}
1-3p & \textrm{if } & j=i, i\in\{1,N\}, \\[7pt]
1-2p & \textrm{if } & j=i, 2\leq i <N,  \\[7pt]
1 & \textrm{if } & j=i=N+1, \\[7pt]
{(2i+1)p\over 2i-1} & \textrm{if } & j=i+1, 1\leq i < N, \\[7pt]
{4Np\over 2N-1} & \textrm{if } & i=N, j=N+1,\\[7pt]
{(2i-3)p\over 2i-1} & \textrm{if } & j=i-1, 1 < i \leq N. \\[7pt]
\end{array}
\right.
$$
%

\noindent
For $d=2N=8$  the transitions are depicted in Fig. \ref{fig:circle_zig} (right).

\tikzstyle{block4} = [fill=gray!0,minimum size=0.4em]

\begin{figure}[H]
\centering
\begin{tabular}{ccc}
 \begin{tikzpicture}[->,>=stealth',shorten >=1pt,node distance=2.5cm,auto,scale=0.85,main node/.style={rectangle,rounded corners,draw,align=center, scale=0.85}]

\def \n {8}
\def \radius {3cm}
\def \radiusB {3.7cm}
\def \radiusC {4cm}

\def \margin {8} 

\draw[dashed, color=gray!80] (0,0) circle (3cm);

\foreach \s in {1,...,\n}
{
\edef\mya{0}
  \pgfmathtruncatemacro{\label}{\s-1} 
  \pgfmathtruncatemacro{\labelA}{\s+4} 
  \pgfmathtruncatemacro{\labelB}{\s-4} 
  
  \ifthenelse{\s=1}{\node[block4] () at ({360/\n * (\n-\s+3)}:\radiusB){$8$}}{};
  \ifthenelse{\s=2}{\node[block4] () at ({360/\n * (\n-\s+3)}:\radiusB){$6$}}{};
  \ifthenelse{\s=3}{\node[block4] () at ({360/\n * (\n-\s+3)}:\radiusB){$4$}}{};
  \ifthenelse{\s=4}{\node[block4] () at ({360/\n * (\n-\s+3)}:\radiusB){$2$}}{};
  \ifthenelse{\s=5}{\node[block4] () at ({360/\n * (\n-\s+3)}:\radiusB){$1$}}{};
  \ifthenelse{\s=6}{\node[block4] () at ({360/\n * (\n-\s+3)}:\radiusB){$3$}}{};
  \ifthenelse{\s=7}{\node[block4] () at ({360/\n * (\n-\s+3)}:\radiusB){$5$}}{};
  \ifthenelse{\s=8}{\node[block4] () at ({360/\n * (\n-\s+3)}:\radiusB){$7$}}{};
  

}

\foreach \s in {1,...,\n}
{
\edef\mya{0}
  \pgfmathtruncatemacro{\label}{\s-1} 
  \pgfmathtruncatemacro{\labelA}{\s+4} 
  \pgfmathtruncatemacro{\labelB}{\s-4} 
  
  \node[block3,draw, circle] (\label) at ({360/\n * (\n-\s+3)}:\radius) 
  {\ifthenelse{\s<5}{$\labelA^*$}{$\labelB^*$}};

}

\node (A1) at ({180}:\radiusC){};
\node (A2) at ({220}:\radiusB){};
\node (A3) at ({140}:\radiusB){};

\draw [->] (A1) .. controls +(0,-2) and +(0,0) .. node [midway, left] {$p$} (A2);
\draw [->] (A1) .. controls +(0,+2) and +(0,0) .. node [midway, left] {$p$} (A3);
\draw [->] (A1) .. controls  +(-0.0,1) and +(-2,1) .. node [midway, above] {$1-2p$} (A1);

\draw[->] (4)   to   node[midway,above] {} (3); 
\draw[->] (3)   to   node[midway,above] {} (5); 
\draw[->] (5)   to   node[midway,above] {} (2); 
\draw[->] (2)   to   node[midway,above] {} (6); 
\draw[->] (6)   to   node[midway,above] {} (1); 
\draw[->] (1)   to   node[midway,above] {} (7); 
\draw[->] (7)   to   node[midway,above] {} (0);

%
%
\end{tikzpicture}
&   
 \begin{tikzpicture}[line join=round,x=1.1cm,y=1.0cm,scale=0.85,every node/.style={scale=0.85}]


 \node[block3,draw, circle] (p1)  at (-2,3.55cm) {$1^*$};
 \node[block3,draw, circle] (p2) at (-0.5,3.55cm) {$2^*$};
 \node[block3,draw, circle] (p3) at (1,3.55cm) {$3^*$};
 \node[block3,draw, circle] (p4) at (2.5,3.55cm) {$4^*$};
 \node[block3,draw, circle] (p5) at (4.0,3.55cm) {$5^*$};

 \node (p_empty) at (0,0cm) {};

\draw[dashed] (p1)   to   node[midway,above] {} (p2);
\draw[dashed] (p2)   to   node[midway,above] {} (p3);
\draw[dashed] (p3)   to   node[midway,above] {} (p4);
\draw[dashed] (p4)   to   node[midway,above] {} (p5);

\draw [->] (p1) .. controls +(0.1,1) and +(-0.1,1) .. node [midway, above] {$3p$} (p2);
\draw [->] (p2) .. controls +(0.1,1) and +(-0.1,1) .. node [midway, above] {${5\over 3}p$} (p3);
\draw [->] (p3) .. controls +(0.1,1) and +(-0.1,1) .. node [midway, above] {${7\over 5}p$} (p4);
\draw [->] (p4) .. controls +(0.1,1) and +(-0.1,1) .. node [midway, above] {${16\over 7}p$} (p5);

\draw [->] (p2) .. controls +(0.1,-1) and +(-0.1,-1) .. node [midway, below] {${1\over 3}p$} (p1);
\draw [->] (p3) .. controls +(0.1,-1) and +(-0.1,-1) .. node [midway, below] {${3\over 5}p$} (p2);
\draw [->] (p4) .. controls +(0.1,-1) and +(-0.1,-1) .. node [midway, below] {${5\over 7}p$} (p3);

 \draw [->] (p2) .. controls +(0.3,2) and +(-0.3,2) .. node [midway, above] {$1-2p$} (p2);
 
 \draw [->] (p3) .. controls +(0.3,2) and +(-0.3,2) .. node [midway, above] {$1-2p$} (p3);
 
 \draw [->] (p4) .. controls +(0.3,2) and +(-0.3,2) .. node [midway, above] {$1-3p$} (p4);

\draw [->] (p5) .. controls +(0.0,1) and +(2,1) .. node [midway, above] {$1$} (p5);

\draw [->] (p1) .. controls  +(-0.0,1) and +(-2,1) .. node [midway, above] {$1-3p$} (p1);
 
 %
%
\end{tikzpicture}
\\ 
$\X$ and the ``Zig-zag'' ordering &      A sharp SSD $\X^* $
\end{tabular}
\caption{Case $d=2N=8$: ``zig-zag'' ordering and state space of $\X$ (left),
the corresponding sharp SSD $\Y^*$ (right) 
}
\label{fig:circle_zig}
\end{figure} 
 
Since there is no confusion (in the chain $\Y^*$), we will identify a state $i^*$ 
simply with $i$. Let $T\equiv T_{1:1:N+1}$ denote the absorption time (in $N+1$) of $\Y^*$ (starting at 1).
Using  Theorem \ref{thm:ETjik_oneAbs} we have:
\begin{equation}\label{eq:circle_ET}
ET_{1:1:N+1}=\sum_{n=1}^{N}\left[d_n\sum_{s=1}^n \frac{1}{p(s)d_s}\right]. 
\end{equation}
Let us write a formula for $p(s)$ explicitly:
\begin{equation}\label{eq:circle_ps}
p(s)=
\left\{ 
\begin{array}{llll}
 {(2s+1)p\over 2s-1} & \textrm{if} & i<N,\\[10pt]
 {4Np\over 2N-1} &\textrm{if} & i=N.
\end{array}
\right.  
\end{equation}
We need to compute $d(s)$. For $1\leq s<N$ we have
$$d_s=\prod_{i=2}^s\frac{q(i)}{p(i)}=\prod_{i=2}^s\frac{\frac{2i-3}{2i-1}}{\frac{2i+1}{2i-1}}=\prod_{i=2}^s\frac{2i-3}{2i+1}=\frac{3}{(2s-1)(2s+1)}$$
and for $s=N$ we have
$$d_N=\prod_{i=2}^N\frac{q(i)}{p(i)}=d_{N-1}\frac{q(N)}{p(N)}=\frac{3}{(2N-3)(2N-1)}\frac{\frac{2N-3}{2N-1}}{\frac{4N}{2N-1}}=\frac{3}{4N(2N-1)}.$$

\noindent 
Plugging above formulas for $p(s), d_s$ in \eqref{eq:circle_ET} 
(and using a formula $\sum_{s=1}^n (2s-1)^2=\frac{n(2n-1)(2n+1)}{3}$) we
obtain for $N>1$:
$ET_{1:1:N+1}=$
$$
\begin{array}{lllll}
& &  \multicolumn{2}{l}{\displaystyle 
\sum_{n=1}^{N}\left[d_n\sum_{s=1}^n \frac{1}{p(s)d_s}\right]
} \\[18pt]
&= & \multicolumn{2}{l}{\displaystyle 
\sum_{n=1}^{N-1}\left[d_n\sum_{s=1}^n \frac{1}{p(s)d_s}\right]+d_N\sum_{s=1}^N \frac{1}{p(s)d_s}
} \\[18pt] 
&= & \multicolumn{2}{l}{\displaystyle 
\sum_{n=1}^{N-1}\left[d_n\sum_{s=1}^n \frac{1}{p(s)d_s}\right]+d_N\sum_{s=1}^{N-1} \frac{1}{p(s)d_s}+\frac{d_N}{p(N)d_N}
} \\[18pt] 
&= & \multicolumn{2}{l}{\displaystyle 
\sum_{n=1}^{N-1}\left[\frac{3}{(2n-1)(2n+1)}\sum_{s=1}^n \frac{1}{p\frac{2s+1}{2s-1}\frac{3}{(2s-1)(2s+1)}}\right]+\frac{3}{4N(2N-1)}\sum_{s=1}^{N-1} \frac{1}{p\frac{2s+1}{2s-1}\frac{3}{(2s-1)(2s+1)}}+\frac{1}{p\frac{4N}{2N-1}}
} \\[18pt] 
&= & \multicolumn{2}{l}{\displaystyle 
\sum_{n=1}^{N-1}\left[\frac{1}{p(2n-1)(2n+1)}\sum_{s=1}^n (2s-1)^2 \right]+\frac{1}{p4N(2N-1)}\sum_{s=1}^{N-1} (2s-1)^2+\frac{{2N-1}}{p4N}
} \\[18pt] 
&= & \multicolumn{2}{l}{\displaystyle 
\sum_{n=1}^{N-1}\left[\frac{1}{p(2n-1)(2n+1)}\frac{n(2n-1)(2n+1)}{3} \right]+\frac{1}{p4N(2N-1)}\frac{(N-1)(2N-3)(2N-1)}{3}+\frac{{2N-1}}{p4N}
} \\[18pt] 
&= & \multicolumn{2}{l}{\displaystyle 
\frac{1}{3p}\sum_{n=1}^{N-1}n +\frac{(N-1)(2N-3)}{p12N}+\frac{{2N-1}}{p4N}
} \\[18pt] 
&= & \multicolumn{2}{l}{\displaystyle 
\frac{4N}{12pN}\frac{N(N-1)}{2} +\frac{(N-1)(2N-3)}{p12N}+\frac{3(2N-1)}{p12N}
} \\[18pt] 
&= & \multicolumn{2}{l}{\displaystyle 
\frac{2N^2(N-1)+(N-1)(2N-3)+3(2N-1)}{12pN} 
} \\[18pt] 
&= & \multicolumn{2}{l}{\displaystyle 
\frac{2N^3-2N^2+2N^2-5N+3+6N-3}{12pN}=\frac{2N^3+N}{12pN}=\frac{2N^2+1}{12p}.
}. 
\end{array}
$$
\item Now consider case $N=1$. \par
We can directly compute a separation distance 
$sep(\nu\PP_X^k,\pi)$ for    $\X$ starting at 1 (i.e., $\nu=(1,0)$).
We have

\begin{equation}\label{eq:fsst_circle_N1_sep}
 sep(\nu\PP_X^k,\pi)=\max_{i\in\{1,2\}}\left(1- {\PP_X^k(1,i)\over {1\over 2}} 
\right)=1-2\min_{i\in\{1,2\}}\PP_X^k(1,i).
\end{equation}
Spectral decomposition yields 
\begin{equation}\label{eq:fsst_circle_N1_spectral_dec}
\PP_X^k={1\over 2} 
\left(
\begin{array}{rrr}
 1  & -1 \\ 
 1 & 1
\end{array}
\right)
\left(
\begin{array}{ll}
 1  & 0 \\ 
 0 & (1-4p)^k
\end{array}
\right)
\left(
\begin{array}{rrr}
 1  & 1 \\ 
 -1 & 1
\end{array}
\right)
={1\over 2}
\left(
\begin{array}{ll}
 1 + (1-4p)^k  & 1-(1-4p)^k \\ 
 1-(1-4p)^k &  1+(1-4p)^k
\end{array}
\right) 
\end{equation}
and thus 
$$sep(\nu\PP_X^k,\pi)= 1-\min\{1 + (1-4p)^k  , 1-(1-4p)^k\} =
\left\{ 
\begin{array}{lll}
 (1-4p)^k & \textrm{ if } & p\in(0,1/4), \\[10pt]
 (4p-1)^k  & \textrm{ if } & p\in (1/4,1/2). 
\end{array}
\right. 
$$
 On the other hand we know that there always exists a fastest strong 
 stationary time $T$ (see Proposition 1.10 (b) in \cite{Diaconis1990a}), i.e.,
 $sep(\nu\PP_X^k,\pi)=P(T>k)$.
 For $p\in(0,1/4)$ we have that $T$ has distribution $Geo(4p)$,
 whereas
 for $p\in(1/4,1/2)$ we have $P(T>k)=(4p-1)^k=(1-2(1-2p))^k$, thus $T$ 
 has distribution $Geo(2(1-2p))$. It implies that 
 $$ET=
\left\{ 
\begin{array}{lll}
 \displaystyle {1\over 4p} & \textrm{ if } & p\in(0,1/4), \\[10pt]
\displaystyle  {1\over 2(1-2p)} & \textrm{ if } & p\in (1/4,1/2). 
\end{array}
\right.  
$$
 
\end{itemize}

\end{proof}

\begin{remark}
\rm For a case $N=1$ and $p\leq 1/4$ we can 
have a duality-based proof, similar to the one we had for  $N>1$. 
From equation \eqref{eq:circle_ps} we have 
$p(1)=p(N)={4p}$,  using Theorem  \ref{thm:ETjik_oneAbs}
we directly have 
$$ET_{1:1:2}=d_1 {1\over p(1)d_1}={1\over p(1)} = {1\over 4p}. $$
 Let us have a closer look at this case.
 Note that   both, a random walk on a  circle  and a resulting 
strong stationary dual, are the chains on two points.
The ordering matrix is given by  
$\C=
\left(
\begin{array}{ll}
 1  & 1 \\ 
 0 & 1
\end{array}
\right)$ and
we   directly have
$$
\PP_X =
\left( 
\begin{array}{ll}
 1-2p & 2p \\ 
 2p & 1-2p
\end{array}
\right),\quad 
\C^{-1}\overleftarrow{\PP}_X\C=
\left( 
\begin{array}{ll}
 1-4p & 2p \\ 
 0 & 1 
\end{array}
\right).
$$
From (\ref{eq:dual_trans}) we   obtain (with $\pi(1)=\pi(2)=1/2$)
$$
\PP^*_X =
\left(
\begin{array}{ll}
 1-4p & 4p \\ 
 0 & 1
\end{array}
\right).
$$
\noindent
The transitions of $\X$ and $\X^*$ are depicted in Figure \ref{fig:circle_2points}.

\begin{figure}[H]
\centering
\begin{tabular}{ccc}
 \begin{tikzpicture}[line join=round,x=1.1cm,y=1.0cm,scale=0.9,every node/.style={scale=0.9}]


 \node[block3,draw, circle] (p1)  at (-2,3.55cm) {$1$};
 \node[block3,draw, circle] (p2) at (-0.5,3.55cm) {$2$};


\draw [->] (p1) .. controls +(0.1,1) and +(-0.1,1) .. node [midway, above] {$2p$} (p2);
\draw [->] (p2) .. controls +(0.0,1) and +(2,1) .. node [midway, above] {$1-2p$} (p2);

\draw [->] (p2) .. controls +(0.1,-1) and +(-0.1,-1) .. node [midway, below] {$2p$} (p1);

\draw [->] (p1) .. controls  +(-0.0,1) and +(-2,1) .. node [midway, above] {$1-2p$} (p1);
 
\end{tikzpicture}
& \ \ \ \ \ &
 \begin{tikzpicture}[line join=round,x=1.1cm,y=1.0cm,scale=0.9,every node/.style={scale=0.9}]


 \node[block3,draw, circle] (p1)  at (-2,3.55cm) {$1$};
 \node[block3,draw, circle] (p2) at (-0.5,3.55cm) {$2$};


\draw [->] (p1) .. controls +(0.1,1) and +(-0.1,1) .. node [midway, above] {$4p$} (p2);
\draw [->] (p2) .. controls +(0.0,1) and +(2,1) .. node [midway, above] {$1$} (p2);

\draw [->] (p1) .. controls  +(-0.0,1) and +(-2,1) .. node [midway, above] {$1-4p$} (p1);
 
\end{tikzpicture}
\\ 
Random walk $\X$* & & A sharp SSD $\X^* $
\end{tabular}
\caption{Case $d=2N=2$: Original random walk on a circle $\X$ (left),
the corresponding sharp SSD $\X^*$ (right)}
\label{fig:circle_2points}
\end{figure} 

\noindent
Of course, time to absorption in $\X^*$ has $Geo(4p)$ distribution, thus $ET={1\over 4p}$.

\end{remark} 

\begin{remark}
 Note that the assumptions on $p$ in Lemma \ref{lem:fsst_circle} (i.e.,
 $p\leq 1/3$ for $N>1$ and casee $p\leq 1/4, p\in(1/4,1/2)$ for $N=1$) are 
 equivalent to non-negativity of the resulting matrix $\PP_X^*$.
 In other words the assumption implies that $\X$ is  ${}^\uparrow$-M\"obius monotne
 (it is \textsl{if and only if} condition).
\end{remark}

%
%


 \section{Proofs}\label{sec:proofs}
 

 \subsection{Gambler's ruin problem, absorbing birth and death chain}
 
 \subsubsection{Proof of Theorems \ref{thm:main_rho_ET} and \ref{thm:ETjik_oneAbs}}\label{sec:proof_gambler_rho_ET}

\begin{proof}[Proof of Theorem \ref{thm:main_rho_ET}]
Consider the birth and death chain  with $j$ and $k$ ($j<k$) as recurrent absorbing states ($p(j)=q(j)=p(k)=q(k)=0$).
First step analysis yields (for $j<i<k$)
\begin{equation}\label{eq:ETji1kA}
ET_{j:i:k}=p(i)(1+ET_{j:i+1:k})+q(i)(1+ET_{j:i-1:k})+(1-q(i)-p(i))(1+ET_{j:i:k}),
\end{equation}
thus
\begin{equation}\label{eq:ETji1k}
ET_{j:i+1:k}=ET_{j:i:k}+\frac{q(i)}{p(i)}\left(ET_{j:i:k}-ET_{j:i-1:k}-\frac{1}{q(i)}\right).
\end{equation}
Since $ET_{j:j:k}=0$, we have:
$$ ET_{j:j+2:k}=ET_{j:j+1:k}\left(1+\frac{q(j+1)}{p(j+1)}\right)-\frac{q(j+1)}{p(j+1)}\frac{1}{q(j+1)}.$$

Recall that $d_s=\prod_{i=j+1}^s \frac{q(i)}{p(i)}$ (where $d_j=1$),  iterating the above equations yields:
\begin{equation}\label{eq:ETTjik}
ET_{j:i:k}=ET_{j:j+1:k}\sum_{s=j}^{i-1}d_s-\sum_{s=j+1}^{i-1}\left[d_s\sum_{m=j+1}^s \frac{1}{p(m)d_m}\right], 
\end{equation}

what can be checked by induction. Plugging (\ref{eq:ETTjik}) into (\ref{eq:ETji1k}) we have:
\par 
$ET_{j:i+1:k}=$
$$
\begin{array}{lllll}
& &  \multicolumn{2}{l}{\displaystyle ET_{j:j+1:k}\sum_{n=j}^{i-1}d_n-\sum_{n=j+1}^{i-1}\left[d_n\sum_{s=j+1}^n \frac{1}{p(s)d_s}\right]} \\[18pt]
&  &  \displaystyle  +\frac{q(i)}{p(i)} & \displaystyle\left(ET_{j:j+1:k}\sum_{n=j}^{i-1}d_n -\sum_{n=j+1}^{i-1}\left[d_n\sum_{s=j+1}^n \frac{1}{p(s)d_s}\right]\right.\\[18pt]
& &  & \displaystyle \quad \left.-ET_{j:j+1:k}\sum_{n=j}^{i-2}d_n-\sum_{n=j+1}^{i-2}[d_n\sum_{m=j+1}^n \frac{1}{p(s)d_s}]-\frac{1}{q(i)}\right)\\[18pt]
&= & \multicolumn{2}{l}{\displaystyle ET_{j:j+1:k}\sum_{n=j}^{i-1}d_n-\sum_{n=j+1}^{i-1}\left[d_n\sum_{s=j+1}^n \frac{1}{p(s)d_s}\right] } \\[18pt] 
&  & \multicolumn{2}{l}{\displaystyle + \ \frac{q(i)}{p(i)}\left(ET_{j:j+1:k}d_{i-1}-d_{i-1}\sum_{s=j+1}^{i-1} \frac{1}{p(s)d_s}-d_i\frac{1}{d_iq(i)}\right)}\\[18pt]
& = & \multicolumn{2}{l}{\displaystyle ET_{j:j+1:k}\sum_{n=j}^{i-1}d_n-\sum_{n=j+1}^{i-1}\left[d_n\sum_{s=j+1}^n \frac{1}{p(s)d_s}\right]+ ET_{j:j+1:k}d_i-d_{i}\sum_{s=j+1}^{i-1} \frac{1}{p(s)d_s}-d_i\frac{1}{d_ip(i)}}\\[18pt]
& = & \multicolumn{2}{l}{\displaystyle ET_{j:j+1:k}\sum_{n=j}^{i}d_n-\sum_{n=j+1}^{i-1}\left[d_n\sum_{s=j+1}^n \frac{1}{p(s)d_s}\right]-d_i\sum_{s=j+1}^{i} \frac{1}{p(s)d_s}}\\[18pt]
& = & \multicolumn{2}{l}{\displaystyle ET_{j:j+1:k}\sum_{n=j}^{i}d_n-\sum_{n=j+1}^{i}\left[d_n\sum_{s=j+1}^n \frac{1}{p(s)d_s}\right]}.

\end{array}
$$
  
%
%
%
%
%
%
%
%
%
%
%
%
%
%
%
%
%
%
%
%

Since $ET_{j:k:k}=0$, we have:
$$\displaystyle 0=ET_{j:j+1:k}\sum_{n=j}^{k-1}d_n-\sum_{n=j+1}^{k-1}\left[d_n\sum_{s=j+1}^n \frac{1}{p(s)d_s}\right] \Rightarrow
ET_{j:j+1:k}=\frac{\displaystyle \sum_{n=j+1}^{k-1}\left[d_n\sum_{s=j+1}^n \frac{1}{p(s)d_s}\right]}{\displaystyle \sum_{n=j}^{k-1}d_n},$$
%
%
thus
$$ET_{j:i:k}=\frac{\sum_{n=j+1}^{k-1}\left[d_n\sum_{s=j+1}^n \frac{1}{p(s)d_s}\right]}{\sum_{n=j}^{k-1}d_n}\sum_{n=j}^{i-1}d_n-\sum_{n=j+1}^{i-1}\left[d_n\sum_{s=j+1}^n \frac{1}{p(s)d_s}\right],$$
what was to be shown.

\end{proof}

 \begin{proof}[Proof of Theorem \ref{thm:ETjik_oneAbs}]
Similarly as to proof of the Theorem \ref{thm:main_rho_ET} we consider birth and death chain 
on $\{j,\ldots,k\}$ ($j<k$), however now only $k$ is absorbing (\ie $p(k)=q(k)=q(j)0$, but $p(j)>0$). 
For $i: j<i<k$ we can rewrite Eq. (\ref{eq:ETji1kA}):

$$ ET_{j:i:k}=p(i)(1+ET_{j:i+1:k})+q(i)(1+ET_{j:i-1:k})+(1-q(i)-p(i))(1+ET_{j:i:k}),$$
we have 
\begin{equation}\label{eq:2ETji1k_C}
ET_{j:i:k}=ET_{j:i+1:k}-\frac{q(i)}{p(i)}\left(ET_{j:i:k}-ET_{j:i-1:k}-\frac{1}{q(i)}\right).
\end{equation}
 However, for $i=j$ we have 
 $$ET_{j:j:k}=(1-p(j))(1+ET_{j:j:k})+p(j)(1+ET_{j:j+1:k}),$$
 \ie 
  $$ET_{j:j:k} = {1\over p(j)}+ET_{j:j+1:k}.$$
Recall that $d_s=\prod_{i=j+1}^s \frac{q(i)}{p(i)}$ (where $d_j=j$),  iterating the above equations yields:
\begin{equation}\label{eq:2ETTjik_C}
ET_{j:i:k}=ET_{j:i+1:k}+\sum_{s=j}^{i}\frac{d_i}{p(s)d_s}, 
\end{equation}
what can be checked by induction.
 Plugging (\ref{eq:2ETTjik_C}) (for $i:=i-1$) into (\ref{eq:2ETji1k_C}) we have:
$$
\begin{array}{lllll}
ET_{j:i:k}&= & \displaystyle 
ET_{j:i+1:k}-\frac{q(i)}{p(i)}\left(ET_{j:i:k}-\left(ET_{j:i:k}+\sum_{s=j}^{i-1}\frac{d_{i-1}}{p(s)d_s}\right)-\frac{1}{q(i)}\right) \\[18pt]

&= & \displaystyle ET_{j:i+1:k}+\frac{d_i}{d_{i-1}}\left(\left(\sum_{s=j}^{i-1}\frac{d_{i-1}}{p(s)d_s}\right)+\frac{d_{i-1}}{p(i)d_i}\right)\\[18pt] 
& = & \displaystyle ET_{j:i+1:k}+\sum_{s=j}^{i}\frac{d_i}{p(s)d_s}.
\end{array}
$$
Since $ET_{j:k:k}=0$, we have:
$$ ET_{j:k-1:k}=\sum_{s=1}^{k-1}\frac{d_{k-1}}{p(s)d_s}.$$
Iterating the above equations yields:
$$ET_{j:i:k}=\sum_{n=i}^{k-1}\left[d_n\sum_{s=1}^n \frac{1}{p(s)d_s}\right],$$
what was to be shown.

\end{proof}

 \subsubsection{ Proof of Lemma \ref{lem:sum} and Theorem \ref{thm:main_cond}}\label{sec:proof_gambler}
  
 \begin{proof}[Proof of Lemma \ref{lem:sum}]  
   Denote by $f(n)$ lhs of (\ref{eq:lem_sum}) and by $h(n)$ its rhs.
   We will show that generating functions of $f$ and $h$ are equal.
   Let us start with   $\mathfrak{g}_f(x)$, the generating function 
    of $f$ at $x$:
  \newpage 
  \smallskip\par 
    $\mathfrak{g}_f(x)=$  
   $$ 
   \begin{array}{llll}
  \displaystyle\sum_{n=0}^\infty f(n)x^n &=&\displaystyle \sum_{n=0}^\infty \sum_{k=0}^n  {n-k\choose k} \left(-{r\over (1+r)^2}\right)^k x^n
  =\sum_{n=0}^\infty  \sum_{k=0}^\infty {n-k\choose k} \left(-{r\over (1+r)^2}\right)^k x^n\\[18pt]
    & =& \displaystyle \sum_{k=0}^\infty  \sum_{n=k}^\infty {n-k\choose k} \left(-{r\over (1+r)^2}\right)^k x^n\\[18pt]
    & = & \displaystyle \sum_{k=0}^\infty  \sum_{n=0}^\infty {n\choose k} \left(-{r\over (1+r)^2}\right)^k x^{n+k}
  = \sum_{k=0}^\infty \left(-{r\over (1+r)^2}\right)^k x^k \sum_{n=0}^\infty {n\choose k} x^{n} 
   \end{array}
$$

   Applying $\displaystyle \sum_{n=0}^\infty {n\choose k} x^{n}= {x^k\over (1-x)^{k+1}}$ we have
   $$ 
   \begin{array}{llll}
   \mathfrak{g}_f(x) & =& \displaystyle\sum_{k=0}^\infty \left(-{r\over (1+r)^2}\right)^k x^k  {x^k\over (1-x)^{k+1}}
   ={1\over 1-x}\sum_{k=0}^\infty \left(- r x^2\over (1+r)^2(1-x) \right)^k \\[18pt]
  &=&\displaystyle {1\over (1-x)} {(1+r)^2(1-x)\over (1+r)^2(1-x)+rx^2}=  {(1+r)^2\over (1+r)^2(1-x)+rx^2}.   
   \end{array}
  $$
   \medskip\par 
   On the other hand,  the generating function of $h$ is following:
   $$ 
   \begin{array}{llll}
  \mathfrak{g}_h(x)
  &   = & \displaystyle \sum_{n=0}^\infty h(n)x^n = \sum_{n=0}^\infty {1-r^{n+1}\over (1+r)^n(1-r)} x^n = {1 \over (1-r)} \left(\sum_{n=0}^\infty {1\over (1+r)^n} x^n  - \sum_{n=0}^\infty {r^n\over (1+r)^n} x^n\right)\\[18pt]
&     =& \displaystyle {1 \over (1-r)} \left(\sum_{n=0}^\infty {1\over (1+r)^n} x^n  - \sum_{n=0}^\infty {r^n\over (1+r)^n} x^n\right)=\displaystyle {1 \over (1-r)} \left({1+r\over 1+r- x}  - r{1+r \over 1+r- xr} \right)\\
    & =& \displaystyle {1 +r\over (1-r)}  {1+r-xr-r-r^2-xr \over (1+r- x)(1+r- xr)} 
      ={1 +r\over (1-r)}  {(1+r)(1-r) \over (1+r)^2-(1+r)(x+xr)+x^2r} \\[18pt]
     &=& \displaystyle {(1+r)^2\over (1+r)^2(1-x)+rx^2},   
\end{array}
  $$
  thus $\mathfrak{g}_h(x)=\mathfrak{g}_f(x)$, what finishes the proof.

  \end{proof}

 The following lemma will be needed in the proof of Theorem \ref{thm:main_cond}.

\begin{lemma}\label{lem:matrixA}
 Consider the gambler's ruin problem with general rates $\p, \q$. 
 Define 
 \begin{eqnarray*} 
a_i & = & 
-{ \rho_{0:i:{i+1}} \over p(i)},\\
b_i & = & \frac{(p(i)+q(i))\rho_{0:i:i+1}}{p(i)},\\
c_i & = &-\frac{q(i)}{p(i)} \rho_{0:i-1:{i+1}}.
\end{eqnarray*}
Then, for all $N\geq 1 $ we have 

\begin{equation*}\label{N2}
 \prod_{j=2}^{N} \begin{pmatrix}
  b_{j} & c_{j}  & a_{j} \\
  1 & 0  & 0 \\
  0 & 0  & 1 
 \end{pmatrix}
 \cdot 
 \begin{pmatrix}
  1 & 0  & a_1 \\
  1 & 0  & 0 \\
  0 & 0  & 1 
 \end{pmatrix}
 = 
 \begin{pmatrix}
1 & 0  & A_{N} \\
1 & 0  & A_{N-1} \\
  0 & 0  & 1 
 \end{pmatrix},
 \end{equation*}
  where
  \begin{equation*}\label{eq:AM}
  A_M=-\sum_{n=1}^M {1\over p(n)} \rho_{0:n:M+1} \sum_{k=0}^{\lfloor (M-n)/2 \rfloor} \xi_k^{n+1,M},
  \end{equation*}
 $\xi_k^{n+1,M}$ was defined in (\ref{eq:xi}).
\end{lemma}
\begin{proof}
 Recall that $\mathbf{j}^{n,m}_{k}$ was defined in (\ref{eq:j}) as
$$ \mathbf{j}^{n,m}_{k}=\big\{\{j_1,j_2,\ldots,j_k\}: j_1 \geq n+1,  j_k \leq m, j_i \leq j_{i+1}-2 \textrm{ for }i \in \{1,k-1\}\big\}.$$
For given $\p, \q$,  $b_n, c_n$     and  $\mathbf{j}\in \mathbf{j}^{n,m}_{k}$ define 
$$D^{n,m}_{\mathbf{j}}=b_nb_{n+1}\ldots b_{j_1-2}c_{j_1}b_{j_1+1}b_{{j_1}+2}\ldots b_{j_2-2}c_{j_2}\ldots b_{j_{k-1}+1}b_{{j_{k-1}}+2}\ldots b_{j_k-2}c_{j_k}b_{j_{k}+1}b_{{j_{k}}+2}\ldots b_{m}$$
and let
$$S^{n,m}_k=\sum_{\mathbf{j} \in \mathbf{j}^{n,m}_k} D^{n,m}_{\mathbf{j}}.$$
Let
\begin{eqnarray*} 
\alpha_i & = &-{ 1 \over p(i)},\\
\beta_i & = & \frac{(p(i)+q(i))}{p(i)}=1+r(i),\\
\gamma_i & = &-\frac{q(i)}{p(i)}=-r(i) .
\end{eqnarray*}
$D^{n,m}_\mathbf{j}$ can be rewritten as
$$
\begin{array}{lll}
D^{n,m}_\mathbf{j}&=&\rho_{0:n:m+1}\beta_n\beta_{n+1}\cdots \beta_{j_1-2}\gamma_{j_1}\beta_{j_1+1}\beta_{{j_1}+2} \cdots \beta_{j_2-2}\gamma_{j_2}\cdots\\[10pt] 
 & &\cdot \beta_{j_{k-1}+1}\beta_{{j_{k-1}}+2}\ldots \beta_{j_k-2}\gamma_{j_k}\beta_{j_{k}+1}\beta_{{j_{k}}+2}\cdots \beta_{m}  \\[10pt]
 & = & \displaystyle(-1)^k\prod_{s \in \mathbf{j}}r(s)\prod_{s \in \{n,\ldots,m\} \setminus \mathbf{j}\cup \mathbf{j}-1  } 1+r(s)=\rho_{0:n:m+1}\delta^{n,m}_\mathbf{j}.\\
\end{array}
$$
Thus $\displaystyle S^{n,m}_k=\sum_{\mathbf{j} \in \mathbf{j}^{n,m}_k} D^{n,m}_{\mathbf{j}}=\rho_{0:n:m+1}\sum_{\mathbf{j} \in \mathbf{j}^{n,m}_k}\delta^{n,m}_\mathbf{j}=:\rho_{0:n:m+1}\xi^{n,m}_k$
and $A_M$ can be rewritten as 
$$A_M=
\sum_{n=1}^M a_n\sum_{k=0}^{\lfloor (M-n)/2 \rfloor} S^{n+1,M}_k.$$
We will show this by induction.
\begin{itemize}
 \item For $M=1$ we have 
 $$A_1=\sum_{n=1}^1 a_n \sum_{k=0}^{\lfloor (1-n)/2 \rfloor} S^{n+1,1}_k=a_1 \sum_{k=0}^{\lfloor 0/2 \rfloor} S^{2,1}_k=a_1S^{2,1}_0=a_1.$$

\item For $N\geq M \geq 2$ assuming  $A_M=\sum_{n=1}^M a_n\sum_{k=0}^{\lfloor (M-n)/2 \rfloor} S^{n+1,M}_k$
we shall prove   that $A_{N+1}=b_{N+1}A_N+c_{N+1}A_{N-1}+a_{N+1}$.
We have \par 
\noindent 
$b_{N+1}A_N+c_{N+1}A_{N-1}+a_{N+1}=$
$$
\begin{array}{lll}
&=& \displaystyle b_{N+1}\sum_{n=1}^N a_n\sum_{k=0}^{\lfloor (N-n)/2 \rfloor} S^{n+1,N}_k+c_{N+1}\sum_{n=1}^{N-1} a_n\sum_{k=0}^{\lfloor (N-n-1)/2 \rfloor} S^{n+1,N-1}_k+a_{N+1}\\[10pt]
&=& \displaystyle \sum_{n=1}^N a_n\sum_{k=0}^{\lfloor (N-n)/2 \rfloor} b_{N+1}\sum_{\mathbf{j}^{n+1,N}_k} D^{n+1,N}_{\mathbf{j}^{n+1,N}_{k}}+\sum_{n=1}^{N-1} a_n\sum_{k=0}^{\lfloor (N-n-1)/2 \rfloor} c_{N+1}\sum_{\mathbf{j}^{n+1,N-1}_k} D^{n+1,N-1}_{\mathbf{j}^{n+1,N-1}_{k}}+a_{N+1}\\[10pt]
&=& \displaystyle  \sum_{n=1}^N a_n\sum_{k=0}^{\lfloor (N+1-n)/2 \rfloor} \sum_{\mathbf{j}^{n+1,N+1}_k:j_k\neq N+1} D^{n+1,N+1}_{\mathbf{j}^{n+1,N+1}_{k}} \\[10pt]
& & + \displaystyle \sum_{n=1}^N a_n\sum_{k=0}^{\lfloor (N+1-n)/2 \rfloor} \sum_{\mathbf{j}^{n+1,N+1}_k:j_k= N+1} D^{n+1,N+1}_{\mathbf{j}^{n+1,N+1}_{k}}+a_{N+1}\\[10pt]
&=& \displaystyle \sum_{n=1}^{N+1} a_n\sum_{k=0}^{\lfloor (N+1-n)/2 \rfloor} \sum_{\mathbf{j}^{n+1,N+1}_k} D^{n+1,N+1}_{\mathbf{j}^{n+1,N+1}_{k}} =\sum_{n=1}^{N+1} a_n\sum_{k=0}^{\lfloor (N+1-n)/2 \rfloor} S^{n+1,N+1}_k=A_{N+1}
\end{array}
$$
\end{itemize}
what finishes the proof. 

\end{proof}

 

\begin{proof}[Proof of Theorem \ref{thm:main_cond}]

%
First step analysis yields (for $N>i>1$):
$$
\begin{array}{llll}
EW_{0:i:N}&=&(1+EW_{0:i-1:N})P(X_1=i-1| X_0=i, X_T=N)\\[6pt]
& &+(1+EW_{0:i:N})P(X_1=i| X_0=i, X_T=N)\\[6pt]
 & & +(1+EW_{0:i+1:N})P(X_1=i+1| X_0=i, X_T=N).
\end{array}
$$
We have $EW_{0:N:N} =0$ and for simplicity we also set $EW_{0:0:N}=0$.
We have
$$
\begin{array}{llll}
P(X_1=i-1| X_0=i, X_T=N)&=&\frac{P(X_1=i-1| X_0=i)P(X_T=N| X_1=i-1)}{P(X_T=N| X_0=i)} =\displaystyle\frac{q(i)\rho_{0:i-1:N}}{\rho_{0:i:N}}=q(i)\rho_{0:i-1:i},\\[8pt]
P(X_1=i| X_0=i, X_T=N) & = &  \displaystyle\frac{(1-p(i)-q(i))\rho_{0:i:N}}{\rho_{0:i:N}}=1-p(i)-q(i),\\[8pt]
P(X_1=i+1| X_0=i, X_T=N) & = & \displaystyle \frac{p(i)\rho_{0:i+1:N}}{\rho_{0:i:N}}=p(i)\rho_{0:i+1:i}.
\end{array}
$$

For $i=1$ we have
$$EW_{0:1:N}=[1+EW_{0:1:N}](1-p(1)-q(1))+[1+EW_{0:2:N}]p(1) \rho_{0:2:1},$$
thus
$$EW_{0:2:N}=\frac{(p(1)+q(1)-1)\rho_{0:1:2}}{p(1)}-1+\frac{(p(1)+q(1))\rho_{0:1:2}}{p(1)}EW_{0:1:N}.$$
For  $1\leq i\leq N $ we have
\begin{equation}\label{eq:EW0iNN}
 EW_{0:i:N}=(1+EW_{0:i-1:N})q(i)\rho_{0:i-1:i}+(1+EW_{0:i:N})(1-p(i)-q(i))+(1+EW_{0:i+1:N})p(i)\rho_{0:i+1:i}
\end{equation}
and
 
\begin{eqnarray} 
 EW_{0:i+1:N}& = &  \displaystyle  \frac{(p(i)+q(i))\rho_{0:i:i+1}}{p(i)}   -\frac{q(i)}{p(i)}\rho_{0:i-1:i+1}  -1-\frac{\rho_{0:i:i+1}}{p(i) }\nonumber \\[12pt]
 & & + \displaystyle \frac{(p(i)+q(i))\rho_{0:i:i+1}}{p(i)}EW_{0:i:N}
-\frac{q(i)}{p(i)}\rho_{0:i-1:i+1} EW_{0:i-1:N}, \label{eq:EW0i1NN} \nonumber \\[10pt]
& = & b_i+c_i-1+a_i+b_i EW_{0:i:N}+c_i EW_{0:i-1:N}\nonumber\\[10pt]
& \stackrel{(*)}= & a_i+b_i EW_{0:i:N}+c_i EW_{0:i-1:N} \label{eq1},
\end{eqnarray}
where $a_i, b_i, c_i$ were defined in Lemma \ref{lem:matrixA} and  in $(*)$ we used the fact that
 
$$
\begin{array}{llll}
 
b_i+c_i
& = &\displaystyle  \frac{(p(i)+q(i))\rho_{0:i:i+1}}{p(i)}-\frac{q(i)}{p(i)} \rho_{0:i:{i+1}}\\[12pt]
& = &\displaystyle \frac{p(i)+q(i)}{p(i)}\frac{\sum_{n=1}^i \prod_{k=1}^{n-1}\left({q(k)\over p(k)}\right)}{
 \sum_{n=1}^{i+1} \prod_{k=1}^{n-1}\left({q(k)\over p(k)}\right)}-\frac{q(i)}{p(i)}\frac{\sum_{n=1}^{i-1} \prod_{k=1}^{n-1}\left({q(k)\over p(k)}\right)}{
 \sum_{n=1}^{i+1} \prod_{k=1}^{n-1}\left({q(k)\over p(k)}\right)}\\[26pt]
& = &\displaystyle \frac{\sum_{n=1}^i \prod_{k=1}^{n-1}\left({q(k)\over p(k)}\right)+\frac{q(i)}{p(i)}\sum_{n=1}^i \prod_{k=1}^{n-1}\left({q(k)\over p(k)}\right)-\frac{q(i)}{p(i)}\sum_{n=1}^{i-1} \prod_{k=1}^{n-1}\left({q(k)\over p(k)}\right)}{\sum_{n=1}^{i+1} \prod_{k=1}^{n-1}\left({q(k)\over p(k)}\right)}\\[26pt]
& = &\displaystyle \frac{\sum_{n=1}^i \prod_{k=1}^{n-1}\left({q(k)\over p(k)}\right)+\frac{q(i)}{p(i)}\sum_{n=i}^i \prod_{k=1}^{n-1}\left({q(k)\over p(k)}\right)}{\sum_{n=1}^{i+1} \prod_{k=1}^{n-1}\left({q(k)\over p(k)}\right)}\\[26pt]
& = &\displaystyle \frac{\sum_{n=1}^i \prod_{k=1}^{n-1}\left({q(k)\over p(k)}\right)+ \prod_{k=1}^{i}\left({q(k)\over p(k)}\right)}{\sum_{n=1}^{i+1} \prod_{k=1}^{n-1}\left({q(k)\over p(k)}\right)}=\displaystyle \frac{\sum_{n=1}^{i+1} \prod_{k=1}^{n-1}\left({q(k)\over p(k)}\right)}{\sum_{n=1}^{i+1} \prod_{k=1}^{n-1}\left({q(k)\over p(k)}\right)}=1. 
\end{array}
$$
 
Equations (\ref{eq:EW0iNN}) and  (\ref{eq1}) can be written in a matrix form:
\begin{equation}\label{eq:matr1}
\begin{pmatrix}
  EW_{0:i+1:N} \\
  EW_{0:i:N} \\
  1 
 \end{pmatrix}
=
\begin{pmatrix}
  b_i & c_i  & a_i \\
  1 & 0  & 0 \\
  0 & 0  & 1 
 \end{pmatrix}
\begin{pmatrix}
  EW_{0:i:N}  \\
  EW_{0:i-1:N} \\
  1
 \end{pmatrix}.
 \end{equation}
Note that $c_1=-\frac{q_1}{p_1}W^{2}_{0}=-\frac{q_1}{p_1}0=0$ and 
$$b_1=\frac{(p(1)+q(1))\rho_{0:1:2}}{p(1)}=\frac{p(1)+q(1)}{p(1)}\frac{\sum_{n=1}^1 \prod_{k=1}^{n-1}\left({q(k)\over p(k)}\right)}{
 \sum_{n=1}^2 \prod_{k=1}^{n-1}\left({q(k)\over p(k)}\right)}
 =\frac{1+\frac{q(1)}{p(1)}}{1}\frac{1}{1+\frac{q(1)}{p(1)}}=1,$$
thus using (\ref{eq:matr1}) recursively we obtain 
\begin{equation*} \label{N1}
\begin{array}{llll}
\begin{pmatrix}
  0 \\
  EW_{0:N-1:N} \\
  1 
 \end{pmatrix}
 =
 \begin{pmatrix}
  EW_{0:N:N} \\
  EW_{0:N-1:N} \\
  1 
 \end{pmatrix}
 &
 =
 &
 \displaystyle\prod_{j=2}^{N-1} \begin{pmatrix}
  b_{j} & c_{j}  & a_{j} \\
  1 & 0  & 0 \\
  0 & 0  & 1 
 \end{pmatrix}
 \cdot 
 \begin{pmatrix}
  1 & 0  & a_1 \\
  1 & 0  & 0 \\
  0 & 0  & 1 
 \end{pmatrix}
\begin{pmatrix}
  EW_{0:1:N}  \\
  EW_{0:0:N} \\
  1
 \end{pmatrix}

 \\[26pt]
 
 & = & \begin{pmatrix}
1 & 0  & A_{N-1} \\
1 & 0  & A_{N-2} \\
  0 & 0  & 1 
 \end{pmatrix}
 \begin{pmatrix}
  EW_{0:1:N}  \\
  0\\
  1
 \end{pmatrix},
\end{array}
\end{equation*}
where $A_{N}$ is given in Lemma \ref{lem:matrixA}, what implies 
$$EW_{0:1:N}=-A_{N-1}$$
and thus proves (\ref{eq:main_gambler_EW1i}). 
Equation (\ref{eq:main_gambler_EWiN}) follows from the fact that $W_{0:1:N} \stackrel{(distr)}= W_{0:1:i}+W_{0:i:N}$
(Markov property,   $W_{0:1:i}$ and $W_{0:i:N}$ are independent).

\end{proof}

 \subsection{Random walk on a polygon}
 \subsubsection{Proof of Theorem \ref{thm:main_polygon}}\label{sec:proof_polygon}



\begin{proof}[Proof of eq. (\ref{thm:main_poly_Ai}) ]
Let $F_i$ denote the event that at the first time we leave state $i$ (recall, ties are allowed) we
move clockwise. Similarly, let  $F_i^c$ denotes the event that at the first time we leave state 
$i$ we move counterclockwise. We have
\begin{eqnarray*}
 P(F_i) & = & {p(i)\over p(i)+q(i)} = \frac{1}{1+r(i)}, \\
 P(F_i^c) & = & {q(i)\over p(i)+q(i)} = \frac{r(i)}{1+r(i)}
\end{eqnarray*}
and
$$P(A_i)=P(F_i)P(A_i|F_i)+P(F_i^c)P(A_i|F_i^c)={1\over 1+r(i)} P(A_i|F_i)+{r(i)\over 1+r(i)} P(A_i|F_i^c).$$
\begin{itemize}
 \item For $P(A_i|F_i)$ we have: we start at $i+1$ and we have to reach $i-1$ before reaching $i$.
 This is the probability of winning in the  game 
 $G(\p,\q,i,i+1,i-1)$.
 We thus have 
 $$P(A_i|F_i) = \rho_{i:i+1:i-1} ={1\over \displaystyle  \sum_{n=i+1}^{i-1} \prod_{s=i+1}^{n-1} r(s)}.$$

\item Similarly for $P(A_i | F_i^c)$ we have: we start at $i-1$, and we have to reach $i+1$ before reaching $i$ which 
corresponds to losing in the game $G(\p,\q,i+1,i-1,i)$.
We thus have

 $$P(A_i|F_i^c) = 1-\rho_{i+1:i-1:i} =1- {\displaystyle  \sum_{n=i+2}^{i-1} \prod_{s=i+2}^{n-1} r(s) \over \displaystyle  \sum_{n=i+2}^i \prod_{s=i+2}^{n-1} r(s)}
={\displaystyle  \prod_{s=i+2}^{i-1} r(s) \over \displaystyle  \sum_{n=i+2}^i \prod_{s=i+2}^{n-1} r(s)} = {1\over \displaystyle  \sum_{n=i+2}^i \prod_{s=n}^{i-1} \left({1\over r(s)}\right)}. $$
\end{itemize}

Finally

\begin{eqnarray*}
 P(A_i ) & = & {1\over \displaystyle (1+r(i)) \sum_{n=i+1}^{i-1} \prod_{s=i+1}^{n-1} r(s)}+ {r(i)\over \displaystyle (1+r(i)) \sum_{n=i+2}^i \prod_{s=n}^{i-1} \left({1\over r(s)}\right)}\label{eq:gen3_1}\\
 & = &{1\over \displaystyle (1+r(i)) \sum_{n=i+1}^{i-1} \prod_{s=i+1}^{n-1} r(s)}+ {1\over \displaystyle (1+r(i)) \sum_{n=i+2}^i \prod_{s=n}^{i} \left({1\over r(s)}\right)}.\nonumber
\end{eqnarray*}

%
%

\end{proof}

\begin{proof}[Proof of  eq. (\ref{thm:main_poly_Lij})] \ 
Let us define $T_1=\inf\{t:X_t=j-1 \vee X_t=j+1 | X_0=i\}$ 
and  consider separately two cases when at  $T_1$ we are at $j-1$ or $j+1$. 
The first one corresponds to winning, whereas the second one 
corresponds to losing in the game  
$G(\p,\q,j+1,i,j-1)$. The winning probability is
$$ \rho_{j+1:i:j-1}.$$
In the first case (when we get to the $j-1$ before $j+1$) vertex $j$ will be the 
last one if we reach $j+1$ earlier  - this can  be interpreted as  
losing in the game $G(\p,\q,j+1,j-1,j)$, what 
happens with  probability: 
\begin{eqnarray*}
 1- \rho_{j+1:j-1:j} & = & 1- {\displaystyle  \sum_{n=j+2}^{j-1} \prod_{s=j+2}^{n-1}r(s) \over 
 \displaystyle   \sum_{n=j+2}^{j} \prod_{s=j+2}^{n-1}r(s)} = {\displaystyle  \prod_{s=j+2}^{j-1}r(s) \over 
 \displaystyle   \sum_{n=j+2}^{j} \prod_{s=j+2}^{n-1}r(s)} .
\end{eqnarray*}
In the second case (when we get to the $j+1$ before $j-1$) vertex $j$ will
be the last one if we reach  $j-1$ earlier - this can be interpreted as 
winning in the game $G(\p,\q,j,j+1,j-1)$, what happens with probability: 
\begin{eqnarray*}
 \rho_{j:j+1:j-1} & = & {\displaystyle  1 \over 
 \displaystyle   \sum_{n=j+1}^{j-1} \prod_{s=j+1}^{n-1}r(s)}.
\end{eqnarray*}
Finally:
{\footnotesize
$$
\begin{array}{llllll}
P(L_{i,j}) & = & (1-\rho_{j+1:i:j-1})\rho_{j:j+1:j-1}+\rho_{j+1:i:j-1}(1- \rho_{j+1:j-1:j})\\[20pt]

 & = & \displaystyle \left(1- { \displaystyle \sum_{n=j+2}^{i} \prod_{s=j+2}^{n-1}r(s) \over 
    \displaystyle\sum_{n=j+2}^{j-1} \prod_{s=j+2}^{n-1}r(s)} \right)
 {\displaystyle  1 \over 
 \displaystyle   \sum_{n=j+1}^{j-1} \prod_{s=j+1}^{n-1}r(s)} +
 {\displaystyle  \sum_{n=j+2}^{i} \prod_{s=j+2}^{n-1}r(s) \over 
 \displaystyle   \sum_{n=j+2}^{j-1} \prod_{s=j+2}^{n-1}r(s)} {\displaystyle  \prod_{s=j+2}^{j-1}r(s) \over 
 \displaystyle   \sum_{n=j+2}^{j} \prod_{s=j+2}^{n-1}r(s)} \\[40pt]
 
  & = &  {\displaystyle  \sum_{n=i+1}^{j-1} \prod_{s=j+2}^{n-1}r(s) \over 
 \displaystyle   \sum_{n=j+2}^{j-1} \prod_{s=j+2}^{n-1}r(s)} 
 {\displaystyle  1 \over 
 \displaystyle   \sum_{n=j+1}^{j-1} \prod_{s=j+1}^{n-1}r(s)} +
 {\displaystyle  \sum_{n=j+2}^{i} \prod_{s=j+2}^{n-1}r(s) \over 
 \displaystyle   \sum_{n=j+2}^{j-1} \prod_{s=j+2}^{n-1}r(s)} {\displaystyle  \prod_{s=j+2}^{j-1}r(s) \over 
 \displaystyle   \sum_{n=j+2}^{j} \prod_{s=j+2}^{n-1}r(s)} \\[40pt]
 
  & = & 
  {\displaystyle  1 \over \displaystyle   \sum_{n=j+2}^{j-1} \prod_{s=j+2}^{n-1}r(s)} \left(
  {\displaystyle  \sum_{n=i+1}^{j-1} \prod_{s=j+2}^{n-1}r(s) \over 
 \displaystyle   \sum_{n=j+1}^{j-1} \prod_{s=j+1}^{n-1}r(s)} +
 {\displaystyle  \left(\sum_{n=j+2}^{i} \prod_{s=j+2}^{n-1}r(s)\right)\left(   \prod_{s=j+2}^{j-1}r(s)\right) \over 
 \displaystyle   \sum_{n=j+2}^{j} \prod_{s=j+2}^{n-1}r(s)} \right) \\[40pt]
 
  & = & 
  {\displaystyle  1 \over \displaystyle   \sum_{n=j+2}^{j-1} \prod_{s=j+2}^{n-1}r(s)} \left(
  {\displaystyle  \sum_{n=i+1}^{j-1} \prod_{s=j+2}^{n-1}r(s) \over 
 \displaystyle   \sum_{n=j+1}^{j-1} \prod_{s=j+1}^{n-1}r(s)} +
 {\displaystyle  \sum_{n=j+2}^{i} \prod_{s=j+2}^{n-1}r(s) \over 
 \displaystyle   \sum_{n=j+2}^{j} \prod_{s=n}^{j-1}\frac{1}{r(s)}} \right).\\

\end{array}
$$
}

\end{proof}

\begin{proof}[Proof of eqs. (\ref{thm:main_poly_Vij}), (\ref{thm:main_poly_Vi}) and (\ref{thm:main_poly_Ri}) ]
Let us start with the expectation of 
$V_{i,j}$ -- number of steps to visit all vertices starting at $i$ when $j$ is the last visited vertex.
As earlier, let  $T_1=\inf\{t:X_t=j-1 \vee X_t=j+1\}$. We have  two cases:\\
\begin{itemize}
\item If $X_{T_1}=j-1$ (and $j$ was the last visited vertex) then the expected game 
time consists of: expected time to win in   $G(\p,\q,j+1,i,j-1)$,
expected time to lose in  $G(\p,\q,j+1,j-1,j)$ and expected duration
of the game $G(\p,\q,j,j+1,j)$. That is:
$$EW_{j+1:i:j-1}  +EB_{j+1:j-1:j}+ ET_{j:j+1:j}$$
\item If $X_{T_1}=j+1$ (and $j$ was last visited vertex) then 
the expected game time consists of: expected time to lose in 
$G(\p,\q,j+1,i,j-1)$, expected time to win in  $G(\p,\q,j,j+1,j-1)$ and 
expected duration  of the game $G(\p,\q,j,j-1,j)$. That is:
$$EB_{j+1:i:j-1} +EW_{j:j+1:j-1}+ET_{j:j-1:j}$$
\end{itemize}

Now, conditioning on $X_{T_1}$, we obtain:

$$EV_{i,j}=\rho_{j+1:i:j-1} \left(EW_{j+1:i:j-1}  +EB_{j+1:j-1:j}+ ET_{j:j+1:j}  \right)$$ $$+ (1-\rho_{j+1:i:j-1} ) \left(EB_{j+1:i:j-1} +EW_{j:j+1:j-1}+ET_{j:j-1:j}\right).$$

Equations (\ref{thm:main_poly_Vi}) and (\ref{thm:main_poly_Ri}) are simply obtained by  conditioning on the states.

 \end{proof}

%
%
%

\bibliographystyle{alpha}
\bibliography{library_cond}


\end{document}